\documentclass[11pt]{article}
\usepackage[english]{babel}		
\usepackage[T1]{fontenc}	
\usepackage{a4wide}
\usepackage{amsmath,amsfonts,amssymb,amsthm,cancel,siunitx,calculator,calc,mathtools,empheq,latexsym}
\usepackage{eucal}
\usepackage{graphicx}
\usepackage[dvipsnames]{xcolor}
\usepackage{subfig,epsfig,tikz,float}	
\usepackage{hyperref}
\hypersetup{colorlinks=true, linkcolor=red, filecolor=blue, urlcolor=cyan}

\theoremstyle{plain}
\newtheorem{thm}{Theorem}
\newtheorem{lem}[thm]{Lemma}

\theoremstyle{definition}
\newtheorem{rem}[thm]{Remark}

\newtheorem{defi}[thm]{Definition}

\numberwithin{thm}{section}
\numberwithin{equation}{section}

\def\supp{\operatorname{supp}}
\def\esup{\operatornamewithlimits{ess\,sup}}

\allowdisplaybreaks

\title{New characterization of weighted inequalities involving
superposition of Hardy integral operators}

\author{%
Amiran Gogatishvili$^{1}$ \ and \ Tu\u{g}\c{c}e \"{U}nver$^{2*}$
}

\begin{document}

\date{}

\maketitle

\vspace{-0.5cm}

\begin{center}
{\footnotesize 
$^1$ Institute of Mathematics of the
 Czech Academy of Sciences,
 \v Zitn\'a~25,
 115~67 Praha~1,
 Czech Republic \\
$^2$Department of Mathematics,
Kirikkale University,
71450 Yahsihan, Kirikkale,
Turkey \\
*Corresponding author. E-mail: tugceunver@kku.edu.tr, \texttt{ORCiD:}{0000-0003-0414-8400}\\
Contributing author: gogatish@math.cas.cz, \texttt{ORCiD:}{0000-0003-3459-0355}
}
\end{center}

\bigskip
\noindent

\bigskip
\noindent
{\small{\bf ABSTRACT.}
Let $1\leq p <\infty$ and $0 < q,r < \infty$. We characterize the validity of the inequality for the composition of the Hardy operator,
\begin{equation*}
\bigg(\int_a^b \bigg(\int_a^x \bigg(\int_a^t f(s)ds \bigg)^q u(t) dt \bigg)^{\frac{r}{q}} w(x) dx  \bigg)^{\frac{1}{r}} \leq C \bigg(\int_a^b f(x)^p v(x) dx \bigg)^{\frac{1}{p}} 
\end{equation*} 
for all non-negative measurable functions $f$ on $(a,b)$, $-\infty \leq a < b \leq \infty$. We construct a more straightforward discretization method than those previously presented in the literature, and we provide some new scales of weight characterizations of this inequality in both discrete and continuous forms and we obtain previous characterizations as the special case of the parameter.}

\medskip
\noindent



\baselineskip=\normalbaselineskip

\section{Introduction and the main results}

Let $-\infty\leq a<b \leq \infty$. Denote by $\mathfrak{M}^+(a,b)$ the set of all non-negative measurable functions on $(a,b)$ and $\mathfrak{M}^{\uparrow}(a,b)$ is the class of non-decreasing elements of $\mathfrak{M}^+(a,b)$. 

In operator theory, weighted inequalities involving operator composition may be found in various topics. Let $0< q,r <\infty$ and $1 \leq p <\infty$. The validity of inequalities
\begin{equation} \label{HC-ineq.}
\bigg(\int_0^{\infty} \bigg(\int_0^x \bigg(\int_t^{\infty} h(s) ds\bigg)^q u(t) dt \bigg)^{\frac{r}{q}} w(x) dx  \bigg)^{\frac{1}{r}} \leq C \bigg(\int_0^{\infty} h(x)^p v(x) dx \bigg)^{\frac{1}{p}}, 
\end{equation} 
and
\begin{equation} \label{HH-ineq.}
\bigg(\int_0^{\infty} \bigg(\int_0^x \bigg(\int_0^t h(s) ds\bigg)^q u(t) dt \bigg)^{\frac{r}{q}} w(x) dx  \bigg)^{\frac{1}{r}} \leq C \bigg(\int_0^{\infty} h(x)^p v(x) dx \bigg)^{\frac{1}{p}},
\end{equation} 
for all $h \in \mathfrak{M}^+(0,\infty)$ are crucial because many classical inequalities can be reduced to them. For example, duality techniques reduce the embeddings between Lorentz-type spaces, Morrey-type spaces, and Ces\'{a}ro-type spaces to the weighted iterated inequalities (see, e.g., \cite{BO, GKPS, GMU-CopCes, GMU-cLMLM, Unver-Ces}). On the other hand, characterizations of weighted bilinear Hardy and Copson inequalities reduce to the characterizations of iterated Hardy inequalities (see, e.g., \cite{AOR, Krepela-EMS, ST}). 

Various approaches have been used to handle inequalities \eqref{HC-ineq.} and \eqref{HH-ineq.} resulting in conditions of a different nature. Inequality \eqref{HC-ineq.} is investigated thoroughly. The most recent results are presented in the paper \cite{KO}. Alternative characterizations of inequalities \eqref{HC-ineq.} and \eqref{HH-ineq.} are found in \cite{K}. Detailed information on the development and history of this inequality may be found in the recent paper \cite{GMPTU}.

Note that, when $q=1$ using Fubini's Theorem, inequality \eqref{HH-ineq.} reduces to the weighted Hardy-type inequality involving kernel, that is,
\begin{equation} \label{kernel-ineq.}
\bigg(\int_0^{\infty} \bigg(\int_0^x \bigg(\int_s^x u(t) dt \bigg) h(s) ds \bigg)^{r} w(x) dx  \bigg)^{\frac{1}{r}} \leq C \bigg(\int_0^{\infty} h(x)^p v(x) dx \bigg)^{\frac{1}{p}}, \quad h \in \mathfrak{M}^{+}(0,\infty).
\end{equation} 
Inequality \eqref{kernel-ineq.} was completely characterized in \cite{MR-Saw89,BloKer91,Oinarov94,Step94} when $1\leq r, p < \infty$.  However, for a long period, there was no adequate characterization in the case when $0 < r < 1 \leq p < \infty$. Several attempts have been made to tackle this case (see, e.g., \cite{Step94,Lai99,Prok2013,GogStep2013}); in some works, necessary and sufficient conditions did not match, while in others, characterization had a discrete form or involved auxiliary functions. Hence, it was not easily verifiable. Finally, in \cite{Krepela-ker}, the missing integral conditions were provided. 

In the general cases, \eqref{HH-ineq.} is characterized in \cite{GogMus-MIA}, but the conditions are in a non-standard form. It was also considered in \cite{ProkStep}, but the conditions are not applicable because they involve auxiliary functions. Recently, in \cite{KrePick-CC}, a more complicated discretization method is used to establish a characterization of the same inequality that involves iteration of the Copson operators and is restricted to non-degenerate weights, and the case $p=1$ is presented without a proof. In our approach, the case $p >1$ is not separated from $p= 1$.

Recently, in \cite{GMPTU}, with a new and simpler discretization technique that requires neither parameter restrictions nor non-degeneracy conditions, characterization of \eqref{HC-ineq.} is given. We adapt this approach to the specific demands of the inequality considered in this paper. Our technique allows us to obtain a scale of characterization which was not possible before; moreover, we obtain the previous characterizations of inequality \eqref{HH-ineq.} as a special case (more, precisely for $\beta =0$, see Theorem~\ref{T:main}). 

We would like to point out that the characterization of  Hardy inequality involving non-decreasing functions, that is,
\begin{equation}\label{monot.ineq}
\bigg(\int_0^{\infty} \bigg(\int_0^x f(s) u(s) ds \bigg)^q  w(x) dx  \bigg)^{\frac{1}{q}} \leq C \bigg(\int_0^{\infty} f(x)^p v(x) dx \bigg)^{\frac{1}{p}},  \quad f \in \mathfrak{M}^{\uparrow}(0,\infty),
\end{equation}
is obtained directly from inequality  \eqref{HH-ineq.} without any further work (see the proof of Theorem~\ref{Cor}) as a direct outcome of our main theorem (see Theorem~\ref{T:main}).

We should also mention that in \cite{GogStep-JMAA,GogStep2013}, using reduction techniques, \eqref{monot.ineq}
is reduced to inequality \eqref{kernel-ineq.}.  However, as we have already mentioned, at that point of time the characterizations of the reduced inequalities were not known. Corollary~3.2 from \cite{GogStep2013} provides a characterization of \eqref{monot.ineq} but the result is non-standard. The earlier works on inequality \eqref{monot.ineq} can be found in \cite{Hein-Step,Myas-Per-Step,Goldman12, Goldman13}.

 We should note that Theorem~\ref{T:main} allowed us to give a scale of characterizations of \eqref{monot.ineq}, and we would like to provide the characterization here to integrate all relevant parameter choices into a single theorem for the reader's convenience (see, Theorem~\ref{Cor}).

As one can see in Section~\ref{Disc.Char.}, the discretization method transforms the inequality at hand equivalently to discrete inequalities that involve local characterizations of inequalities having low-order iterations. For this very reason, our aim in this paper is to revisit inequality \eqref{HH-ineq.} on $(a,b)$ where $-\infty\leq a<b \leq \infty$.

Let $-\infty\leq a<b \leq \infty$ and a weight be a positive measurable function on $(a, b)$. The principal goal of this study is to determine the scale of necessary and sufficient conditions on weights $u,v,w$ on $(a,b)$ for which
\begin{equation} \label{main-iterated}
\bigg(\int_a^b \bigg(\int_a^x \bigg(\int_a^t f(s)ds \bigg)^q u(t) dt \bigg)^{\frac{r}{q}} w(x) dx  \bigg)^{\frac{1}{r}} \leq C \bigg(\int_a^b f(x)^p v(x) dx \bigg)^{\frac{1}{p}}
\end{equation} 
holds for $f\in \mathfrak{M}^+(a,b)$, with exponents $1 \leq p < \infty$ and $0 < q,r < \infty$. It is worth noting that if $p<1$, inequality \eqref{main-iterated} only holds for trivial functions. 

Let us first review the essential notations and conventions before presenting our main results. The left and right sides of the inequality numbered by $(*)$ are denoted by LHS$(*)$ and RHS$(*)$, respectively. We put $0.\infty =  \infty/\infty= 0/0 =0$. The symbol $A \lesssim B$ means a constant $c>0$ exists such that $A \leq c B$ where $c$ depends only on the parameters $p,q,r$. If both $A \lesssim B$ and $B \lesssim A$, then we write $A \approx B$.

For  $1\leq p < \infty$, non-negative measurable functions $v$ and $w$ on $(a,b)$, and $x,y \in [a,b]$, denote by
\begin{align}\label{Vp} 
	V_p(x, y) := \left\{ \begin{array}{ccc}
		\big(\int_x^y v^{-\frac{1}{p-1}}\big)^{\frac{p-1}{p}}, & 1<p<\infty,  \\
			\esup\limits_{s \in (x, y)} v(s)^{-1}, & p=1,	\end{array}
	\right.
\end{align}
and
\begin{equation*}
    \mathcal{W}(x) = \int_x^b w(s)  \,ds.
\end{equation*}

Now, we are ready to formulate our main results. 

\begin{thm}\label{T:main}
Let $1 \leq p < \infty$, $0 < q, r < \infty$, $-\infty<\beta<1$ and let $u, v, w$ be weights on $(a,b)$ such that
$0< \mathcal{W}(t) <\infty$ for all $t\in(a,b)$. Then inequality \eqref{main-iterated} holds for all $f \in \mathfrak{M}^+(a, b)$ if and only if 

\rm(i) $p \leq r $, $p\leq q$ and
\begin{equation}\label{C1}
C_1 :=  \esup_{x \in (a, b)} \bigg(\int_x^b \mathcal{W}(t)^{-\beta} w(t)  \bigg(\int_x^t \mathcal{W}^{\frac{\beta q}{r}}u \bigg)^{\frac{r}{q}} dt \bigg)^{\frac{1}{r}} V_p(a, x) < \infty.
\end{equation}
Moreover, the best constant $C$ in inequality \eqref{main-iterated} satisfies $C \approx C_1$.

\rm(ii) $r < p \leq q$,
\begin{equation*}
C_2 :=\bigg(\int_a^b  \mathcal{W}(x)^{\frac{(1-\beta)p}{p-r} -1} w(x) \esup_{t \in (a, x)} \bigg(\int_t^{x} \mathcal{W}^{\frac{\beta q}{r}} u \bigg)^{\frac{rp}{q(p-r)}} V_p(a, t)^{\frac{pr}{p-r}} dx \bigg)^{\frac{p-r}{pr}}  < \infty,
\end{equation*}
and
\begin{align}\label{C3}
C_3 &:= \bigg( \int_a^b \bigg(\int_x^b \mathcal{W}(s)^{-\beta} w(s) \bigg(\int_{x}^s \mathcal{W}^{\frac{\beta q}{r}}u \bigg)^{\frac{r}{q}} ds \bigg)^{\frac{r}{p-r}} \mathcal{W}(x)^{-\beta}w(x)  \notag \\
 &\hskip+2cm \times \esup_{t \in (a, x)} \bigg(\int_t^{x} \mathcal{W}^{\frac{\beta q}{r}} u \bigg)^{\frac{r}{q}} V_p(a, t)^{\frac{pr}{p-r}} dx \bigg)^{\frac{p-r}{pr}}  < \infty.
\end{align}
Moreover, the best constant $C$ in inequality \eqref{main-iterated} satisfies $C \approx C_2 + C_3$.

\rm(iii) $q< p \leq r $, $C_1 < \infty$ and
\begin{equation*}
	C_4 := \sup_{x \in (a, b)} \mathcal{W}(x)^{\frac{1-\beta}{r}}  \bigg( \int_a^x \bigg(\int_{t}^x \mathcal{W}^{\frac{\beta q}{r}}u \bigg)^{\frac{q}{p-q}} \mathcal{W}(t)^{\frac{\beta q}{r}}u(t) V_p(a, t)^{\frac{pq}{p-q}} dt  \bigg)^{\frac{p-q}{pq}} < \infty,
\end{equation*}
where $C_1$ is defined in \eqref{C1}. Moreover, the best constant $C$ in inequality \eqref{main-iterated} satisfies $C \approx C_1 + C_4$.

\rm(iv) $r < p$, $q < p$, $C_3 < \infty$ and  
\begin{align*}
	C_5 &:= \bigg( \int_a^b \mathcal{W}(x)^{\frac{(1-\beta)p}{p-r} -1} w(x)\\
 &\hskip+2cm\times\bigg( \int_a^x \bigg(\int_{t}^x \mathcal{W}^{\frac{\beta q}{r}}u \bigg)^{\frac{q}{p-q}} \mathcal{W}(t)^{\frac{\beta q}{r}}u(t) V_p(a, t)^{\frac{pq}{p-q}} dt \bigg)^{\frac{r(p-q)}{q(p-r)}} dx \bigg)^{\frac{p-r}{pr}}
	< \infty,
\end{align*}
where $C_3$ is defined in \eqref{C3}. Moreover, the best constant $C$ in inequality \eqref{main-iterated} satisfies $C \approx C_3 + C_5$.
\end{thm}

\begin{thm}\label{Cor}
Let $0 < p,q <\infty$, $-\infty<\beta <1$ and $u,v,w$ be weights on $(a,b)$ such that $0< \mathcal{W}(t) <\infty$ for all $t\in(a,b)$. Then inequality 
\begin{equation}\label{monot.ineq-ab}
\bigg(\int_a^b \bigg(\int_a^x f(s) u(s) ds \bigg)^q  w(x) dx  \bigg)^{\frac{1}{q}} \leq C \bigg(\int_a^b f(x)^p v(x) dx \bigg)^{\frac{1}{p}}
\end{equation}
holds for all $f \in \mathfrak{M}^{\uparrow}(a,b)$ if and only if 

\rm(i) $p \leq q $, $p\leq 1$ and
\begin{equation}\label{C:C1}
	\mathcal{C}_1 :=  \esup_{x \in (a, b)} \bigg(\int_x^b \mathcal{W}(t)^{-\beta} w(t)  \bigg(\int_x^t \mathcal{W}^{\frac{\beta}{q}}u \bigg)^q dt \bigg)^{\frac{1}{q}} \bigg(\int_x^b v \bigg)^{-\frac{1}{p}} < \infty.
\end{equation}
Moreover, the best constant $C$ in inequality \eqref{monot.ineq-ab} satisfies $C \approx \mathcal{C}_1$.

\rm(ii) $q < p \leq 1$,
\begin{align*}
\mathcal{C}_2 &:=\bigg(\int_a^b \mathcal{W}(x)^{\frac{(1-\beta)p}{p-q}-1}w(x) \esup_{t \in (a, x)} \bigg(\int_t^{x} \mathcal{W}^{\frac{\beta}{q}}u \bigg)^{\frac{pq}{p-q}} \bigg(\int_t^b v \bigg)^{-\frac{q}{p-q}} dx \bigg)^{\frac{p-q}{pq}}  < \infty,
\end{align*}
and
\begin{align}\label{C:C3}
\mathcal{C}_3 &:= \bigg( \int_a^b \bigg(\int_x^b \mathcal{W}(s)^{-\beta}w(s) \bigg(\int_{x}^s \mathcal{W}^{\frac{\beta}{q}}u \bigg)^q ds \bigg)^{\frac{q}{p-q}} \mathcal{W}(x)^{-\beta} w(x)   \notag \\
&\hskip+2cm\times\esup_{t \in (a, x)} \bigg(\int_t^{x} \mathcal{W}^{\frac{\beta}{q}}u \bigg)^q \bigg(\int_t^b v \bigg)^{-\frac{q}{p-q}} dx \bigg)^{\frac{p-q}{pq}}  < \infty.
\end{align}
Moreover, the best constant $C$ in inequality \eqref{monot.ineq-ab} satisfies $C \approx \mathcal{C}_2 + \mathcal{C}_3$.

\rm(iii) $1 < p \leq q $, $C_1 < \infty$ and
\begin{equation*}
\mathcal{C}_4 := \sup_{x \in (a, b)} \mathcal{W}(x)^{\frac{1-\beta}{q}} \bigg( \int_a^x \bigg(\int_{t}^x \mathcal{W}^{\frac{\beta}{q}} u \bigg)^{\frac{1}{p-1}} \mathcal{W}(t)^{\frac{\beta}{q}} u(t) \bigg(\int_t^b v \bigg)^{-\frac{1}{p-1}} dt  \bigg)^{\frac{p-1}{p}} < \infty,
\end{equation*}
where $\mathcal{C}_1$ is defined in \eqref{C:C1}. Moreover, the best constant $C$ in inequality \eqref{monot.ineq-ab} satisfies $C \approx \mathcal{C}_1 + \mathcal{C}_4$.

\rm(iv) $q < p$, $1 < p$, $\mathcal{C}_3 < \infty$ and  
\begin{align*}
\mathcal{C}_5 &:= \bigg( \int_a^b \mathcal{W}^{\frac{(1-\beta)p}{p-q}-1} w(x) \\
& \hskip+2cm \times\bigg( \int_a^x \bigg(\int_{t}^x \mathcal{W}^{\frac{\beta}{q}} u \bigg)^{\frac{1}{p-1}} \mathcal{W}(t)^{\frac{\beta}{q}} u(t) \bigg(\int_t^b v \bigg)^{-\frac{1}{p-1}} dt \bigg)^{\frac{q(p-1)}{p-q}} dx \bigg)^{\frac{p-q}{pq}} < \infty,
\end{align*}
where $\mathcal{C}_3$ is defined in \eqref{C:C3}. Moreover, the best constant $C$ in inequality \eqref{monot.ineq-ab} satisfies $C \approx \mathcal{C}_3 + \mathcal{C}_5$.
\end{thm}

Proofs of Theorem~\ref{T:main} and Theorem~\ref{Cor} will be given in  Section~\ref{S:Proofs}.

\section{Preliminary Results}

In this section, we cover the foundations of discretization and several new results that will be employed frequently throughout the proof of the main theorem.

\begin{defi}
Let $N \in \mathbb{Z} \cup \{-\infty\}$, $M \in \mathbb{Z} \cup \{+\infty\}$, $N<M$, and $\{a_k\}_{k=N}^M$ be a sequence of positive numbers. We say that $\{a_k\}_{k=N}^M$ is \textit{geometrically decreasing} if
\begin{equation*}
\sup\bigg\{\frac{a_{k+1}}{a_k},\quad N\leq k \leq M\bigg\} < 1.
\end{equation*}
\end{defi}

\begin{lem}\cite{GogPick2003}\label{L:dec-equiv-1}
Let $\alpha >0$ and $n \in \mathbb{Z}\cup\{-\infty\}$.  If $\{\tau_k\}_{k=n}^{\infty}$ is a  geometrically decreasing sequence,  then,
\begin{equation}
\sup_{n\leq k<\infty} \tau_k \sum_{i=n}^k a_i  \approx \sup_{n\leq k <\infty} \tau_k a_k, \label{sup-sum}
\end{equation} 
\begin{equation}
\sum_{k=n}^{\infty} \tau_k  \bigg(\sum_{i=n}^k a_i \bigg)^{\alpha}  \approx \sum_{k=n}^{\infty} \tau_k a_k^{\alpha}, \label{sum-sum}
\end{equation}
and
\begin{equation}
\sum_{k=n}^{\infty} \tau_k   \sup_{n \leq i \leq k} a_i  \approx \sum_{k=n}^{\infty} \tau_k  a_k, \label{sum-sup}
\end{equation}
for all non-negative sequences $\{a_k\}_{k=n}^{\infty} $.
\end{lem}

\begin{lem} \label{dec-equiv}
Let $\alpha >0$ and $n  \in \mathbb{Z}\cup\{-\infty\}$.  Assume that $\{x_k\}_{k=n}^{\infty} \subset (a,b)$ is a strictly increasing sequence. If $\{\tau_k\}_{k=n+1}^{\infty}$ is a  geometrically decreasing sequence,  then,
\begin{equation}
\sup_{n+1\leq k <\infty} \tau_k \bigg( \int_{x_{n}}^ {x_k} g \bigg) \approx \sup_{n +1\leq k <\infty} \tau_k \bigg( \int_{x_{k-1}}^{x_{k}} g\bigg), \label{dec-sup-sum}
\end{equation}
\begin{equation}
\sum_{k=n+1}^{\infty} \tau_k  \bigg( \int_{x_{n}}^{x_k} g \bigg)^{\alpha}  \approx \sum_{k=n+1}^{\infty} \tau_k  \bigg( \int_{x_{k-1}}^{x_k} g \bigg)^{\alpha}, \label{dec-sum-sum}
\end{equation}
and
\begin{equation}
\sum_{k=n+1}^{\infty} \tau_k   \esup_{s \in (x_{n}, x_k)}  g(s)  \approx \sum_{k=n+1}^{\infty} \tau_k  \esup_{s \in (x_{k-1}, x_{k})}  g(s), \label{dec-sum-sup}
\end{equation}
for all non-negative measurable $g$ on $(a,b)$.
\end{lem}
\begin{proof}
For each $n \in \mathbb{Z} \cup \{-\infty\}$, we can write 
$$
\int_{x_{n}}^{x_k} g = \sum_{i=n+1}^k \int_{x_{i-1}}^{x_i} g.
$$   
Then \eqref{dec-sup-sum} and \eqref{dec-sum-sum} are direct consequences of \eqref{sup-sum} and \eqref{sum-sum}, respectively. 

Similarly, for each $n \in \mathbb{Z} \cup \{-\infty\}$, we have
$$
\esup_{s \in (x_{n}, x_k)}  g(s) = \sup_{n+1 \leq i \leq k} \esup_{s \in (x_{i-1}, x_i)}  g(s),
$$
so that applying \eqref{sum-sup}, we obtain \eqref{dec-sum-sup}.
\end{proof}

\begin{lem}
Let  $\alpha >0$ and $n \in \mathbb{Z}\cup\{-\infty\}$. Assume that $\{x_k\}_{k=n}^{\infty}\subset (a,b)$ is a strictly increasing sequence, $\{\tau_k\}_{k=n+1}^{\infty}$ is a geometrically decreasing sequence, and $\{\sigma_k\}_{k=n}^{\infty}$ is a positive non-decreasing sequence. Then 
\begin{equation}\label{3-sup-equiv}
\sup_{n+1 \leq k <\infty}  \tau_k  \sup_{n \leq i < k} \bigg(\int_{x_i}^{x_k} g \bigg)^{\alpha} \sigma_i \approx \sup_{n+1 \leq k <\infty} \tau_k  \bigg(\int_{x_{k-1}}^{x_k} g\bigg)^{\alpha} \sigma_{k-1}
\end{equation}
and
\begin{equation}\label{3-sum-equiv}
\sum_{k=n+1}^{\infty} \tau_k  \sup_{n \leq i < k} \bigg(\int_{x_i}^{x_k} g\bigg)^{\alpha} \sigma_i \approx \sum_{k=n+1}^{\infty} \tau_k  \bigg(\int_{x_{k-1}}^{x_k} g\bigg)^{\alpha} \sigma_{k-1}
\end{equation}
hold for all non-negative measurable $g$ on $(a,b)$.
\end{lem}

\begin{proof}
Let us start with the equivalency \eqref{3-sup-equiv}. Since $\{\tau_k\}_{k=n+1}^{\infty}$ is a geometrically decreasing sequence, interchanging supremum and \eqref{dec-sup-sum} give 
\begin{equation*}
LHS\eqref{3-sup-equiv} = \sup_{n+1 \leq i <\infty} \sigma_i \sup_{i+1 \leq k <\infty} \tau_k \bigg(\int_{x_i}^{x_k} g\bigg)^{\alpha} \approx    \sup_{n+1 \leq i<\infty} \sigma_i \sup_{i+1 \leq k<\infty} \tau_k \bigg(\int_{x_{k-1}}^{x_k} g\bigg)^{\alpha}.
\end{equation*}
Interchanging supremum once again and monotonicity of $\{\sigma_k\}_{k=n}^{\infty}$ results in 
\begin{equation*}
LHS\eqref{3-sup-equiv} \approx \sup_{n+1 \leq k<\infty} \tau_k  \bigg(\int_{x_{k-1}}^{x_k} g\bigg)^{\alpha} \sup_{n\leq i \leq k-1} \sigma_i = RHS\eqref{3-sup-equiv}. 
\end{equation*}

Let us now tackle \eqref{3-sum-equiv}. Monotonicity of $\{\sigma_k\}_{k=n}^{\infty}$ gives that
\begin{align*}
LHS\eqref{3-sum-equiv} \leq \sum_{k=n+1}^{\infty} \tau_k  \sup_{n\leq i < k} \bigg(\sum_{j=i}^{k-1}  \sigma_j^{\frac{1}{\alpha}}  \int_{x_j}^{x_{j+1}} g \bigg)^{\alpha} = \sum_{k=n+1}^{\infty} \tau_k  \bigg(\sum_{j=n}^{k-1}  \sigma_j^{\frac{1}{\alpha}}  \int_{x_j}^{x_{j+1}} g \bigg)^{\alpha}.
\end{align*}
Then, using \eqref{sum-sum}, we have the following upper estimate
\begin{align*}
LHS\eqref{3-sum-equiv} \leq \sum_{k=n+1}^{\infty} \tau_k  \bigg(\sum_{j=n+1}^{k}  \sigma_{j-1}^{\frac{1}{\alpha}} \int_{x_{j-1}}^{x_{j}} g \bigg)^{\alpha} \approx RHS\eqref{3-sum-equiv}.
\end{align*}
On the other hand, the reverse estimate is clear, and the proof is complete. 
\end{proof}

\begin{defi}
Let $w$ be a non-negative measurable function on $(a,b)$  such that
$0< \mathcal{W}(t) <\infty$ for all $t\in(a,b)$. A strictly increasing sequence  $\{x_k\}_{k=N}^{\infty}\subset [a,b]$ is said to be a discretizing sequence of the function $ {\mathcal W}$, if it satisfies $ {\mathcal W}(x_k)  \approx 2^{-k}$, $N \leq k <  \infty $. If $N > -\infty$ then $x_{N} : =a$, otherwise  $x_{-\infty}:= \lim_{k\rightarrow -\infty} x_k   = a$.	
\end{defi}

\begin{rem}
Let $w$ be a non-negative measurable function on $(a,b)$ such that $0< \mathcal{W}(t) <\infty$ for all $t\in(a,b)$. The Darboux property of continuous functions implies that the discretizing sequence exists for $\mathcal{W}$. We can define the discretizing sequence in the following way: 
\begin{equation*}
N = \sup\bigg\{k\in \mathbb{Z} \colon 2^{-k} \geq \mathcal{W}(t) \,\,  \textit{for every} \,\, t \in (a,b) \bigg\}.
\end{equation*}
If the set is empty then $N = -\infty$. If $\lim_{t \to a+} \mathcal{W}(t) <\infty$, then $N>-\infty$ and $x_N = a$
and for $k> N$
\begin{equation*}
x_k = \sup\bigg\{t \in (a,b) \colon 2^{-k} \leq \mathcal{W}(t)  \bigg\}.   \end{equation*}
\end{rem}

It is worth noting that if $N=-\infty$, then $N+1$ is also $-\infty$.

\begin{lem}
Let $\alpha > 0$ and $N \in \mathbb{Z}\cup \{-\infty\}$. Assume that $w$ is a weight on $(a, b)$ such that $0< \mathcal{W}(t) <\infty$ for all $t\in(a,b)$,
and  $\{x_k\}_{k=N}^{\infty}$ is a discretizing sequence of the function ${\mathcal W}$. Then for any $n \colon N\leq n$, 
\begin{equation}\label{int.equiv}
\int_{x_n}^b {\mathcal W}(x)^{\alpha -1} w(x) h(x)  dx \approx \sum_{k=n+1}^{\infty} 2^{-k \alpha } h(x_k)
\end{equation}
and
\begin{equation}\label{sup.equiv}
\esup_{x \in (x_n,b)}  {\mathcal W}(x)^{\alpha} h(x) \approx \sup_{n+1\leq k} 2^{-k\alpha} h(x_k)
\end{equation}
hold for all non-negative and non-decreasing $h$ on $(a, b)$. 
\end{lem}
\begin{proof}
Monotonicity of $h$ and properties of the discretizing sequence $\{x_k\}_{k=N}^{\infty}$ yield
\begin{align*}
LHS\eqref{int.equiv} &= \sum_{k=n+1}^{\infty} \int_{x_{k-1}}^{x_k} h(x)  {\mathcal W}(x)^{\alpha-1} w(x) dx  \lesssim \sum_{k=n+1}^{\infty} h(x_k) \int_{x_{k-1}}^{x_k}  d\Big[-{\mathcal W}(x)^{\alpha} \Big] \\
& \approx \sum_{k=n+1}^{\infty} 2^{-k\alpha} h(x_k) = RHS\eqref{int.equiv}, 
\end{align*}
and, conversely
\begin{align*}
LHS\eqref{int.equiv} & \geq \sum_{k=n+1}^{\infty} \int_{x_k}^{x_{k+1}} h(x) {\mathcal W}(x)^{\alpha-1} w(x) dx  \gtrsim \sum_{k=n+1}^{\infty} h(x_k) \int_{x_k}^{x_{k+1}}  d\Big[-{\mathcal W}(x)^{\alpha} \Big] \\
& \approx \sum_{k=n+1}^{\infty} 2^{-k\alpha}  h(x_k) = RHS\eqref{int.equiv}.
\end{align*}
Thus, \eqref{int.equiv} holds. 

On the other hand, similarly, 
\begin{align*}
LHS\eqref{sup.equiv} =	\sup_{n+1 \leq k < \infty} \esup_{x \in (x_{k-1},x_k)}  {\mathcal W}(x)^{\alpha} h(x) \approx \sup_{n+1 \leq k < \infty} 2^{-k\alpha} \esup_{x \in (x_{k-1},x_k)}  h(x) = RHS\eqref{sup.equiv}
\end{align*}
holds.
\end{proof}

\begin{thm} \cite[Theorem~4.5]{GPU-JFA}	\label{T:disc.lp.emb.}
Let $p,q \in (0,\infty)$, $N, M \in {\mathbb{Z}}\cup\{-\infty, \infty\}$, $N<M$. Assume that $\{w_k\}_{k=N}^M$ and $\{v_k\}_{k=N}^M$ be two sequences of positive real numbers. Then inequality
\begin{equation}\label{Landau}
    \left(\sum_{k=N}^M a_k^q v_k^q \right)^{\frac{1}{q}} \leq
    C \left(\sum_{k=N}^M a_k^p w_k^p \right)^{\frac{1}{p}}
\end{equation}
holds for every sequence  $\{a_k\}_{k=N}^M$  of non-negative real numbers if and only if either

{\rm(i)} $p \leq q $ and
\begin{equation*}
    L_1= \sup_{N \leq k \leq M} \left\{v_k w_k^{-1} \right\}< \infty,
\end{equation*}
or

{\rm(ii)} $p > q$ and
\begin{equation*}
    L_2= \left(\sum_{k=N}^M v_k^{\frac{pq}{p-q}} w_k^{-\frac{pq}{p-q}} \right)^{\frac{p-q}{pq}} < \infty.
\end{equation*}
Moreover, if we denote by $C$ the optimal embedding constant in~\eqref{Landau}, then
\begin{equation*}
    C =
        \begin{cases}
            L_1 &\text{in the case \textup{(i)},}
                \\
            L_2 &\text{in the case \textup{(ii)}.}
        \end{cases}
\end{equation*}
\end{thm}

\begin{thm} \cite[Theorem~4.6]{GPU-JFA}\label{T:disc.hardy}
Assume that $0 < p,q < \infty$, $N, M \in {\mathbb{Z}}\cup\{-\infty, \infty\}$, $N<M$,   $\{a_k\}_{k=N}^M$ and
$\{b_k\}_{k=N}^M$ are sequences of non-negative real numbers. Then the inequality
\begin{equation}\label{E:disc.hardy}
\left( \sum_{k=N}^M \left(     \sum_{i =N}^k x_i b_i \right)^q a_k \right)^{\frac{1}{q}} \leq C \left(\sum_{k=N}^M x_k^p \right)^{\frac{1}{p}}
\end{equation}
holds for every sequence $\{x_k\}_{k=N}^M$ of non-negative real numbers if and only if
	
{\rm(i)} either $p \leq 1$, $p\leq q $ and
\begin{equation*}
    H_1= \sup_{N \leq k \leq M} \left\{ \left(\sum_{i =k}^{M} a_i \right)^{\frac{1}{q}} b_k\right\} < \infty,
\end{equation*}

{\rm(ii)} or $q < p \le 1$ and
\begin{equation*}
    H_2= \left( \sum_{k=N}^M a_k \left(\sum_{i =k}^{M} a_i\right)^{\frac{q}{p-q}} \sup_{N \leq i\leq k} b_i^{\frac{qp}{p-q}}  \right)^{\frac{p-q}{pq}} < \infty,
\end{equation*}

{\rm(iii)} or $1 < p $, $q < p$ and
\begin{equation*}
    H_3= \left( \sum_{k=N}^M a_k \left(\sum_{i = k}^M a_i\right)^{\frac{q}{p-q}} 	\left(\sum_{i= N}^{k} b_i^{\frac{p}{p-1}}\right)^{\frac{q(p-1)}{p-q}} \right)^{\frac{p-q}{pq}}< \infty,
\end{equation*}

{\rm(iv)} or $ 1< p \le q $ and
\begin{equation*}
    H_4=  \sup_{N\leq k \leq M} \left\{ \left(\sum_{i = k}^{M} a_i\right)^{\frac{1}{q}} \left(\sum_{i= N}^{k} b_i^{\frac{p}{p-1}}\right)^{\frac{p-1}{p}} \right\}< \infty.
\end{equation*}
Moreover, if we denote by $C$ the best constant in~\eqref{E:disc.hardy}, then
\begin{equation*}
    C \approx
        \begin{cases}
            H_1 &\text{in the case \textup{(i)},}
                \\
            H_2 &\text{in the case \textup{(ii)},}
                \\
            H_3 &\text{in the case \textup{(iii)},}
                \\
            H_4 &\text{in the case \textup{(iv)}.}
        \end{cases}
\end{equation*}
Moreover, the multiplicative constants in all the equivalences above depend only on $p,q$.
\end{thm}

\section{Discrete Characterization}\label{Disc.Char.}

We begin this section by observing that inequality \eqref{main-iterated} is equivalent to two other discrete inequalities, and we present the characterization in discrete form, which is noteworthy on its own.

Let us start with the discretization of inequality \eqref{main-iterated}. To this end we need the following notation, denote by $M(x_{k-1}, x_k)$ the best constant of weighted Hardy inequality, that is, \begin{equation*}\label{B(a,b)}
B(x_{k-1},x_k) := \sup_{h\in  \mathfrak{M}^+(x_{k-1},x_k)} \frac{\bigg(\int_{x_{k-1}}^{x_k} \bigg(\int_{x_{k-1}}^t h(s) ds \bigg)^q u(t) dt \bigg)^{\frac{1}{q}}}{\bigg(\int_{x_{k-1}}^{x_k} h(t)^p v(t) dt\bigg)^{\frac{1}{p}}}
\end{equation*}
and using the classical characterizations of weighted Hardy inequalities (see, \cite{KufPerSam,SinStep}), we have for $1\leq p <\infty, 0<q < \infty$
\begin{equation}\label{B-char.}
B(x_{k-1}, x_k) \approx\begin{cases}
\displaystyle \esup\limits_{t \in (x_{k-1}, x_{k})} \bigg(\int_t^{x_{k}} u \bigg)^{\frac{1}{q}} V_p(x_{k-1}, t) &\text{if} \quad p \leq q,\\
\\
\displaystyle 
\bigg(\int_{x_{k-1}}^{x_{k}} \bigg(\int_t^{x_{k}} u \bigg)^{\frac{q}{p-q}} u(t) V_p(x_{k-1}, t)^{\frac{pq}{p-q}}  dt \bigg)^{\frac{p-q}{pq}}  &\text{if} \quad q <  p.        \end{cases}
\end{equation}

Furthermore, observe that $V_p(x_{k-1}, x_k)$, $N+1 \leq k$ defined in \eqref{Vp} can be expressed as
\begin{eqnarray*} \label{Vp-def}
\sup_{g \in  \mathfrak{M}^+(x_{k-1}, x_k)} \frac{\int_{x_{k-1}}^{x_k} g(t) dt} {\bigg(\int_{x_{k-1}}^{x_k} g(t)^p v(t) dt \bigg)^{\frac{1}{p}}} = V_p(x_{k-1}, x_k).
\end{eqnarray*}

\begin{lem} \label{T:equiv.ineq.-0}
Let $1\leq   p <\infty$, $0 < q, r < \infty$ and let $u, v, w$ be weights on $(a,b)$ such that $0< \mathcal{W}(t) <\infty$ for all $t\in(a,b)$, and $\{x_k\}_{k=N}^{\infty}\subset [a,b]$ is a discretizing sequence of the function ${\mathcal W}$. Then, there exists a positive constant $C$ such that inequality \eqref{main-iterated} holds for all $f \in \mathfrak{M}^+(a,b)$ if and only if there exist positive constants $\mathcal{C}'$ and $\mathcal{C}''$ such that
\begin{equation}\label{main-iterated-1}
\bigg( \sum_{k=N+1}^{\infty}  2^{-k} \bigg( \int_{x_{k-1}}^{x_k} \bigg(\int_{x_{k-1}}^t f \bigg)^q u(t) dt \bigg)^{\frac{r}{q}}  \bigg)^{\frac{1}{r}} \leq \mathcal{C}' \bigg( \sum_{k=N+1}^{\infty} \int_{x_{k-1}}^{x_k} f^p v\bigg)^{\frac{1}{p}}
\end{equation}	
and
\begin{equation} \label{main-iterated-2}
\bigg( \sum_{k=N+1}^{\infty}  2^{-k}   \bigg( \int_a^{x_k} f \bigg)^r \bigg( \int_{x_k}^{x_{k+1}}  u \bigg)^{\frac{r}{q}}  \bigg)^{\frac{1}{r}}  \leq \mathcal{C}'' \bigg(\sum_{k=N+1}^{\infty} \int_{x_{k-1}}^{x_k} f^p v \bigg)^{\frac{1}{p}}
\end{equation}	
hold for all $f \in \mathfrak{M}^+(a,b)$. Moreover, $C \approx \mathcal{C}' + \mathcal{C}'' $.
\end{lem}
\begin{proof}
Applying \eqref{int.equiv} with $\alpha =1$ and $h(x) = \bigg(\int_a^{x} \bigg(\int_a^t f \bigg)^q u(t) dt \bigg)^{\frac{r}{q}}$, we have that
\begin{equation*}
LHS\eqref{main-iterated} \approx \bigg( \sum_{k=N+1}^{\infty}  2^{-k} \bigg(\int_a^{x_k} \bigg(\int_a^t f \bigg)^q u(t) dt \bigg)^{\frac{r}{q}} \bigg)^{\frac{1}{r}}.
\end{equation*}
Since, $\{2^{-k}\}$ is geometrically decreasing, using \eqref{dec-sum-sum}, we obtain that
\begin{align*}
LHS\eqref{main-iterated} & \approx \bigg(  \sum_{k=N+1}^{\infty}   2^{-k} \bigg( \int_{x_{k-1}}^{x_k} \bigg(\int_a^t f\bigg)^q u(t) dt \bigg)^{\frac{r}{q}} \bigg)^{\frac{1}{r}} \\
& \approx \bigg( \sum_{k=N+1}^{\infty}    2^{-k} \bigg( \int_{x_{k-1}}^{x_k} \bigg(\int_{x_{k-1}}^t f\bigg)^q u(t) dt \bigg)^{\frac{r}{q}} \bigg)^{\frac{1}{r}} \\
&\quad + \bigg(  \sum_{k=N+2}^{\infty}    2^{-k} \bigg(\int_a^{x_{k-1}} f \bigg)^r \bigg( \int_{x_{k-1}}^{x_k}  u\bigg)^{\frac{r}{q}} \bigg)^{\frac{1}{r}}.
\end{align*}	
Then, it is clear that there exists a positive constant $C$ such that inequality \eqref{main-iterated} holds for all $f \in \mathfrak{M}^+(a,b)$ if and only if there exist positive constants $\mathcal{C}'$ and $\mathcal{C}''$ such that 
\eqref{main-iterated-1} and \eqref{main-iterated-2}
hold for all $f \in \mathfrak{M}^+(a,b)$. Moreover, $C \approx \mathcal{C}' + \mathcal{C}'' $.
\end{proof}

\begin{lem}\label{T:equiv.ineq.-1}
Let $1\leq   p <\infty$, $0 < q, r < \infty$ and let $u, v$ be weights on $(a,b)$. Assume that  $N\in {\mathbb{Z}}\cup\{-\infty\}$ and $\{x_k\}_{k=N}^{\infty}$ is a strictly increasing sequence in $[a, b]$. Then there exists a positive constant $\mathcal{C}'$ such that inequality \eqref{main-iterated-1}
holds for all $f \in \mathfrak{M}^+(a,b)$ if and only if there exists a positive constant $C'$ such that
\begin{equation}\label{Bk-inequality}
\bigg(\sum_{k=N+1}^{\infty} 2^{-k} a_k^r B(x_{k-1},x_{k})^r \bigg)^{\frac{1}{r}} \leq C'  \bigg(\sum_{k=N+1}^{\infty} a_k^p \bigg)^{\frac{1}{p}}
\end{equation}
holds for every sequence of non-negative numbers $\{a_k\}_{k=N+1}^{\infty}$. Moreover the best constants $\mathcal{C}'$ and $C'$, respectively, in \eqref{main-iterated-1} and \eqref{Bk-inequality} satisfy $\mathcal{C}' \approx C'$.
\end{lem}	

\begin{proof}
 Assume that \eqref{main-iterated-1} holds. By the definition of $B(x_{k-1}, x_k)$, there exist non-negative measurable functions $h_k$, $N+1 \leq k$ on $(a, b)$ such that 
$$
\supp h_k \subset [x_{k-1}, x_k], \quad \int_{x_{k-1}}^{x_k} h_k^p v = 1, \quad  \bigg(\int_{x_{k-1}}^{x_k} \bigg(\int_{x_{k-1}}^t h_k \bigg)^q u(t) dt\bigg)^{\frac{1}{q}} \gtrsim B(x_{k-1}, x_k).
$$
Thus, inserting $f = \sum_{m=N+1}^{\infty}  a_m h_m$,  where $\{a_m\}_{m=N+1}^{\infty}$ is any sequence of non-negative numbers, into \eqref{main-iterated-1}, \eqref{Bk-inequality} follows. Moreover, $C' \lesssim \mathcal{C}'$.

Conversely, observe first that for each $h\in \mathfrak{M}^+(x_{k-1},x_k)$,
\begin{equation*}
    B(x_{k-1},x_k) \geq \bigg(\int_{x_{k-1}}^{x_k} \bigg(\int_{x_{k-1}}^t h(s) ds \bigg)^q u(t) dt \bigg)^{\frac{1}{q}} \bigg(\int_{x_{k-1}}^{x_k} h(t)^p v(t) dt\bigg)^{-\frac{1}{p}}.
\end{equation*}
Then \eqref{main-iterated-1} follows by, inserting  $a_k = (\int_{x_{k-1}}^{x_k} h^pv)^{\frac{1}{p}}$ in \eqref{Bk-inequality} and $\mathcal{C}' \leq C'$. Further, we have $\mathcal{C}' \approx C'$.
\end{proof}

\begin{lem}\label{T:equiv.ineq.-2}
Let $1 \leq p <\infty$, $0 < q, r < \infty$ and let $u, v$ be weights on $(a,b)$. Assume that  $N\in {\mathbb{Z}}\cup\{-\infty \}$ and $\{x_k\}_{k=N}^{\infty}$ is a strictly increasing sequence in $[a, b]$. Then there exists a positive constant $\mathcal{C}''$ such that inequality \eqref{main-iterated-2} holds for all $f \in \mathfrak{M}^+(a,b)$ if and only if there exists a positive constant $C''$ such that
\begin{equation}\label{Vp-inequality}
\bigg(\sum_{k=N+1}^{\infty} 2^{-k} \bigg(\int_{x_k}^{x_{k+1}} u \bigg)^{\frac{r}{q}} \bigg(\sum_{j=N+1}^k a_j\, V_p(x_{j-1}, x_j) \bigg)^r \, \bigg)^{\frac{1}{r}} \leq C'' \bigg(\sum_{k=N+1}^{\infty} a_k^p\bigg)^{\frac{1}{p}}
\end{equation}	
holds for every sequence of non-negative numbers $\{a_k\}_{k=N+1}^{\infty}$. Moreover the best constants $\mathcal{C}''$ and $C''$, respectively, in \eqref{main-iterated-2} and \eqref{Vp-inequality} satisfy $\mathcal{C}'' \approx C''$.
\end{lem}	

\begin{proof}
Suppose that  \eqref{main-iterated-2} holds for all $f \in \mathfrak{M}^+(a,b)$. 
Using \eqref{Vp-inequality}, there exist non-negative measurable functions $g_k$, $N+1 \leq k$ on $(x_{k-1}, x_k)$ such that 
$$
\supp g_k \subset [x_{k-1}, x_{k}], \quad \int_{x_{k-1}}^{x_{k}} g_k^p v = 1, \quad  \int_{x_{k-1}}^{x_{k}} g_k  \gtrsim V_p(x_{k-1}, x_k).
$$
Thus, inserting $f = \sum_{m=N+1}^{\infty}  a_m g_m$,   where $\{a_m\}_{m=N+1}^{\infty}$ is any sequence of non-negative numbers, into \eqref{main-iterated-2}, \eqref{Vp-inequality} follows. Moreover, $C'' \lesssim \mathcal{C}''$ holds. 

Conversely, taking $a_k = (\int_{x_{k-1}}^{x_k} f^pv)^{\frac{1}{p}}$ in \eqref{Vp-inequality} and using the estimate  
\[ V_p(x_{k-1}, x_k)\ge \int_{x_{k-1}}^{x_k} f(t) dt\bigg(\int_{x_{k-1}}^{x_k} f(t)^p v(t) dt \bigg)^{-\frac{1}{p}} \]
gives \eqref{main-iterated-2}. Additionally,  $\mathcal{C}'' \leq C''$ holds. Consequently  $\mathcal{C}'' \approx C''$ follows. 
\end{proof}

\begin{thm}\label{T:disc.char.(3.3)}
Let $1 \leq p < \infty$, $0 < q, r < \infty$, $-\infty < \beta<1$ and let $u, v, w$ be weights on $(a,b)$ such that $0< \mathcal{W}(t) <\infty$ for all $t\in(a,b)$, and  $\{x_k\}_{k=N+1}^{\infty}$ be the discretizing sequence of ${\mathcal W}$. Then inequality \eqref{Bk-inequality} holds for every sequence of non-negative numbers $\{a_k\}_{k=N+1}^{\infty}$ if and only if 

\rm(i) $p \le r$, $p\leq q$  and 
\begin{equation*}
\mathcal{A}_1 :=\sup_{N+1\leq  k <\infty} 2^{-\frac{k(1-\beta)}{r}}  \esup_{t \in (x_{k-1}, x_{k})} \bigg(\int_t^{x_{k}} {\mathcal W}^{\frac{\beta q}{r}}u \bigg)^{\frac{1}{q}} V_p(x_{k-1}, t)  < \infty,
\end{equation*}

\rm(ii) $r < p\leq q$ and
\begin{equation*}
\mathcal{A}_2 :=\bigg(\sum_{k=N+1}^{\infty} 2^{-k(1-\beta)\frac{p}{p-r}} \esup_{t \in (x_{k-1}, x_{k})} \bigg(\int_t^{x_{k}} {\mathcal W}^{\frac{\beta q}{r}}u\bigg)^{\frac{pr}{q(p-r)}} 
V_p(x_{k-1}, t)^{\frac{pr}{p-r}} \bigg)^{\frac{p-r}{pr}} < \infty,
\end{equation*}

\rm(iii) $q< p \le r $ and
\begin{equation*}
\mathcal{A}_3 :=\sup_{N+1 \leq k <\infty} 2^{-\frac{k(1-\beta)}{r}} \bigg(\int_{x_{k-1}}^{x_{k}} \bigg(\int_t^{x_{k}} {\mathcal W}^{\frac{\beta q}{r}} u \bigg)^{\frac{q}{p-q}} {\mathcal W}(t)^{\frac{\beta q}{r}}u(t) V_p(x_{k-1}, t)^{\frac{pq}{p-q}}  dt \bigg)^{\frac{p-q}{pq}} < \infty,
\end{equation*}

\rm(iv) $r < p$, $q < p$ and  
\begin{equation*}
\mathcal{A}_4 :=\bigg( \sum_{k=N+1}^{\infty} 2^{-k(1-\beta)\frac{p}{p-r}} \bigg( \int_{x_{k-1}}^{x_{k}} \bigg(\int_{t}^{x_{k}} {\mathcal W}^{\frac{\beta q}{r}}u \bigg)^{\frac{q}{p-q}} {\mathcal W}(t)^{\frac{\beta q}{r}}u(t)  V_p(x_{k-1}, t)^{\frac{pq}{p-q}}  dt \bigg)^{\frac{r(p-q)}{q(p-r)}} \bigg)^{\frac{p-r}{pr}} < \infty.
\end{equation*}
Moreover, the best constant $C'$ in inequality \eqref{Bk-inequality} satisfies 
\begin{equation*}
C' =\begin{cases}
\mathcal{A}_1  &\text{in the case \textup{(i)},}\\
\mathcal{A}_2  &\text{in the case \textup{(ii)},}\\
\mathcal{A}_3 &\text{in the case \textup{(iii)},}\\
\mathcal{A}_4 &\text{in the case \textup{(iv)}}.
\end{cases}
\end{equation*}
\end{thm}

\begin{proof}
The result follows easily by combining Theorem~\ref{T:disc.lp.emb.}     
for suitable parameters and weights with  \eqref{B-char.} and the fact that $\mathcal{W}(t)\approx 2^{-k}$, for every $t \in [x_{k-1}, x_k]$.
\end{proof} 
\begin{thm}\label{T:disc.char.(3.4)}
Let $1 \leq p < \infty$, $0 < q, r < \infty$, $-\infty<\beta<1$ and let $u, v, w$ be weights on $(a,b)$  such that
$0< \mathcal{W}(t) <\infty$ for all $t\in(a,b)$, and  $\{x_k\}_{k=N+1}^{\infty}$ be the discretizing sequence of ${\mathcal W}$. Then inequality \eqref{Vp-inequality} holds for every sequence of non-negative numbers $\{a_k\}_{k=N+1}^{\infty}$if and only if 

\rm(i) $p \le r$ and
\begin{equation*}\label{B1*}
\mathcal{B}_1 :=  \sup_{N+1 \leq k <\infty} \bigg(\sum_{i=k}^{\infty} 2^{-i(1-\beta)} \bigg(\int_{x_i}^{x_{i+1}} {\mathcal W}^{\frac{\beta q}{r}}u \bigg)^{\frac{r}{q}} \bigg)^{\frac{1}{r}} V_p(a, x_k) < \infty,
\end{equation*}

\rm(ii) $r < p$ and
\begin{align*}
\mathcal{B}_2 &:= \bigg( \sum_{k=N+1}^{\infty} 2^{-k(1-\beta)} \bigg(\int_{x_k}^{x_{k+1}} {\mathcal W}^{\frac{\beta q}{r}} u \bigg)^{\frac{r}{q}} \notag \\
& \hspace{2cm} \times \bigg(\sum_{i=k }^{\infty} 2^{-i(1-\beta)} \bigg(\int_{x_i}^{x_{i+1}} {\mathcal W}^{\frac{\beta q}{r}} u \bigg)^{\frac{r}{q}}  \bigg)^{\frac{r}{p-r}} V_p(a, x_k)^{\frac{pr}{p-r}} \bigg)^{\frac{p-r}{pr}}< \infty. \label{B2*}
\end{align*}
Moreover, the best constant $C'' $ in inequality \eqref{Vp-inequality} satisfies 
\begin{equation*}
C'' =\begin{cases}
\mathcal{B}_1  &\text{in the case \textup{(i)},}\\
\mathcal{B}_2  &\text{in the case \textup{(ii)}.}
        \end{cases}
\end{equation*}

\end{thm}

\begin{proof}
We will apply Theorem~\ref{T:disc.hardy}  with suitable parameters  and weights $a_k=2^{-k}\big(\int_{x_k}^{x_{k+1}} u \big)^{\frac{r}{q}} $ and $b_k = V_p(x_{k-1},x_k)$. We need to treat the cases when $p>1$ and $p=1$, separately. Note that when $p>1$, since $\{x_k\}_{k=N+1}^{\infty}$ is the discretizing sequence of ${\mathcal W}$, we have for each $k: N+1 \leq k$ that
\begin{align*}
\sum_{i=N+1}^{\infty} V_p(x_{i-1},x_i)^{\frac{p}{p-1}} = \sum_{i=N+1}^{\infty} \int_{x_{i-1}}^{x_i} v(s)^{-\frac{1}{p-1}} ds = V_p(a,x_k)^{\frac{p-1}{p}}. 
\end{align*}
Thus, applying [Theorem~\ref{T:disc.hardy}, (iv)] when $1<p\leq r$ and [Theorem~\ref{T:disc.hardy}, (iii)] when $1<p, r<p$, the result follows. 
 
On the other hand, if $p=1$, for each $k: N+1 \leq k$,
\begin{equation*}
\sup_{N+1\leq i \leq k} V_1(x_{i-1},x_i) =  \sup_{N+1\leq i \leq k} \esup_{s\in(x_{i-1},x_i)} v(s)^{-1} =  V_1(a,x_k).  
\end{equation*}
Therefore, applying [Theorem~\ref{T:disc.hardy}, (ii)] when $r<p=1$, we have $C'' \approx \mathcal{B}_2$. 

Lastly, if $p=1\leq r$, then  applying [Theorem~\ref{T:disc.hardy}, (i)], and the fact that $\mathcal{W}(t)\approx 2^{-i}$, for every $t \in [x_{i}, x_{i+1}]$ we have
\begin{align*}
C''  &\approx \sup_{N+1 \leq k <\infty} \bigg(\sum_{i=k}^{\infty} 2^{-i} \bigg(\int_{x_i}^{x_{i+1}} u \bigg)^{\frac{r}{q}} \bigg)^{\frac{1}{r}} V_1(x_{k-1}, x_k)\\
&\approx \sup_{N+1 \leq k <\infty} \bigg(\sum_{i=k}^{\infty} 2^{-i(1-\beta)} \bigg(\int_{x_i}^{x_{i+1}} {\mathcal W}^{\frac{\beta q}{r}}u \bigg)^{\frac{r}{q}} \bigg)^{\frac{1}{r}} V_1(x_{k-1}, x_k).
\end{align*}
Finally, interchanging supremum yields that
\begin{align*}
C''&\approx \sup_{N+1 \leq k <\infty} V_1(x_{k-1}, x_k) \sup_{k\leq m<\infty} \bigg(\sum_{i=m}^{\infty} 2^{-i(1-\beta)} \bigg(\int_{x_i}^{x_{i+1}} {\mathcal W}^{\frac{\beta q}{r}} u \bigg)^{\frac{r}{q}} \bigg)^{\frac{1}{r}} \\
& = \sup_{N+1 \leq m <\infty} \bigg(\sum_{i=m}^{\infty} 2^{-i(1-\beta)} \bigg(\int_{x_i}^{x_{i+1}}{\mathcal W}^{\frac{\beta q}{r}} u \bigg)^{\frac{r}{q}} \bigg)^{\frac{1}{r}}  \sup_{N+1\leq k \leq  m} V_1(x_{k-1}, x_k) = \mathcal{B}_1.
\end{align*}

\end{proof}
We can formulate the discrete characterization of inequality \eqref{main-iterated}.

\begin{thm}\label{T:disc.char.}
Let $1 \leq p < \infty$, $0 < q, r < \infty$ and let $u, v, w$ be weights on $(a,b)$ such that $0< \mathcal{W}(t) <\infty$ for all $t\in(a,b)$, and  $\{x_k\}_{k=N+1}^{\infty}$ be the discretizing sequence of ${\mathcal W}$. Then inequality \eqref{main-iterated} holds for all $f \in \mathfrak{M}^+(a,b)$ if and only if 

\rm(i) $p \le r$, $p\leq q$ and $
\mathcal{A}_1 + \mathcal{B}_1 <\infty$,

\rm(ii) $r < p\leq q$ and
$\mathcal{A}_2 + \mathcal{B}_2<\infty$,

\rm(iii) $q< p \le r $ and $\mathcal{A}_3 + \mathcal{B}_1<\infty$,

\rm(iv) $r < p$, $q < p$ and $\mathcal{A}_4 +\mathcal{B}_2 <\infty$.

Moreover the best constant $C$ in \eqref{main-iterated} satisfies
\begin{equation*}
C =\begin{cases}
\mathcal{A}_1 +\mathcal{B}_1  &\text{in the case \textup{(i)},}\\
\mathcal{A}_2 +\mathcal{B}_2  &\text{in the case \textup{(ii)},}\\
\mathcal{A}_3+\mathcal{B}_1  &\text{in the case \textup{(iii)},}\\
\mathcal{A}_4 +\mathcal{B}_2  &\text{in the case \textup{(iv)}.}
\end{cases}
\end{equation*}
\end{thm}

\begin{proof}
According to Lemmas~[\ref{T:equiv.ineq.-0}-\ref{T:equiv.ineq.-2}], inequality \eqref{main-iterated} holds if and only if inequalities \eqref{Bk-inequality} and \eqref{Vp-inequality} hold. Moreover, the best constant $C$ in \eqref{main-iterated} satisfies $C\approx C' + C''$, where $C'$ and $C''$ are the best constants in inequalities \eqref{Bk-inequality} and \eqref{Vp-inequality}, respectively. The result follows from the combination of Theorem~\ref{T:disc.char.(3.3)} and Theorem~\ref{T:disc.char.(3.4)}.
\end{proof}

\section{Proofs}\label{S:Proofs}

\textbf{Proof of Theorem~\ref{T:main}}

\rm(i) Let $1 \leq p \leq \min\{r,q\}$. We have from [Theorem~\ref{T:disc.char.}, (i)] that $C \approx \mathcal{A}_1 + \mathcal{B}_1$. We will prove that $C_1 \approx \mathcal{A}_1 + \mathcal{B}_1$. First, we will show that $\mathcal{A}_1 + \mathcal{B}_1 \approx A_1 + \mathcal{B}_1$, where
\begin{equation*}
A_1 := \sup_{N+1 \leq  k} 2^{-\frac{k(1-\beta)}{r}}  \esup_{t \in (a, x_{k})} \bigg(\int_t^{x_{k}} {\mathcal W}^{\frac{\beta q}{r}}u \bigg)^{\frac{1}{q}} V_p(a, t).
\end{equation*}

It is clear that $\mathcal{A}_1 \leq A_1$. On the other hand, observe that 
\begin{align*}
A_1 & = \sup_{N+1 \leq  k} 2^{-\frac{k(1-\beta)}{r}} \sup_{N+1\leq i \leq k} \esup_{t \in (x_{i-1}, x_{i})} \bigg(\int_t^{x_k} {\mathcal W}^{\frac{\beta q}{r}}u \bigg)^{\frac{1}{q}} V_p(a, t)\\
& \approx \sup_{N+1 \leq  k} 2^{-\frac{k(1-\beta)}{r}} \sup_{N+1\leq i \leq k} \esup_{t \in (x_{i-1}, x_{i})} \bigg(\int_t^{x_i} {\mathcal W}^{\frac{\beta q}{r}} u \bigg)^{\frac{1}{q}} V_p(a, t) \\
& \quad + \sup_{N+2\leq  k} 2^{-\frac{k(1-\beta)}{r}} \sup_{N+1\leq i < k} \bigg(\int_{x_i}^{x_k} {\mathcal W}^{\frac{\beta q}{r}}u \bigg)^{\frac{1}{q}} V_p(a, x_i).
\end{align*}
Then, interchanging the supremum in the first term and applying  \eqref{3-sup-equiv} with $n=N+2$, for the second term, we have that
\begin{align*}
A_1 & \approx  \sup_{N+1 \leq  k} 2^{-\frac{k(1-\beta)}{r}} \esup_{t \in (x_{k-1}, x_{k})} \bigg(\int_t^{x_k} {\mathcal W}^{\frac{\beta q}{r}} u \bigg)^{\frac{1}{q}} V_p(a, t) \\
& \hspace{2cm} + \sup_{N+2\leq  k} 2^{-\frac{k(1-\beta)}{r}} \bigg(\int_{x_{k-1}}^{x_k} {\mathcal W}^{\frac{\beta q}{r}}u \bigg)^{\frac{1}{q}} V_p(a, x_{k-1}).
\end{align*}
Note that, for any $k \geq N+2$, we have
\begin{equation}\label{V-cut}
V_p(a,t) \approx V_p(a, x_{k-1}) + V_p(x_{k-1}, t), \quad \text{for every} \quad t\in (x_{k-1}, x_k).
\end{equation}
Then, in view of \eqref{V-cut}, 
\begin{align*}
A_1 & \approx  \sup_{N+1 \leq  k} 2^{-\frac{k(1-\beta)}{r}} \esup_{t \in (x_{k-1}, x_{k})} \bigg(\int_t^{x_k} {\mathcal W}^{\frac{\beta q}{r}}u \bigg)^{\frac{1}{q}} V_p(x_{k-1}, t) \\
&\hskip+1cm + \sup_{N+2 \leq  k} 2^{-\frac{k(1-\beta)}{r}} \bigg(\int_{x_{k-1}}^{x_k} {\mathcal W}^{\frac{\beta q}{r}}u \bigg)^{\frac{1}{q}} V_p(a, x_{k-1})\\
& \lesssim \mathcal{A}_1 + \mathcal{B}_1.
\end{align*}
Then we have that $\mathcal{A}_1 + \mathcal{B}_1 \leq A_1 + \mathcal{B}_1 \lesssim \mathcal{A}_1 + \mathcal{B}_1$. 

It remains to show that $A_1 + \mathcal{B}_1 \approx C_1$. Applying \eqref{sup.equiv} with $\alpha = \frac{1-\beta}{r}$,
\begin{equation*}
A_1 \approx \esup_{x \in (a, b)} \bigg(\int_x^b w \bigg)^{\frac{1-\beta}{r}} \esup_{t\in (a,x)} \bigg(\int_t^x {\mathcal W}^{\frac{\beta q}{r}}u \bigg)^{\frac{1}{q}} V_p(a, t)
\end{equation*}
holds, and interchanging supremum gives that
\begin{align}\label{A1<C1}
A_1 &\approx \esup_{t \in (a, b)} V_p(a, t) \esup_{x\in (t,b)} \bigg(\int_x^b w \bigg)^{\frac{1-\beta}{r}}  \bigg(\int_t^x {\mathcal W}^{\frac{\beta q}{r}}u \bigg)^{\frac{1}{q}} \notag\\ 
& \leq \esup_{t \in (a, b)} V_p(a, t) \bigg(\int_t^b {\mathcal W}(s)^{-\beta }w(s) \bigg(\int_t^s {\mathcal W}^{\frac{\beta q}{r}} u \bigg)^{\frac{r}{q}} ds \bigg)^{\frac{1}{r}}= C_1.
\end{align}
On the other hand, applying \eqref{int.equiv} with $\alpha= 1-\beta$, then using \eqref{dec-sum-sum} with $n=k$, we obtain for any $k\geq N$ that
\begin{align}\label{int-wu-sum}
\int_{x_k}^b {\mathcal W(s)}^{-\beta} w(s) \bigg(\int_{x_k}^s {\mathcal W}^{\frac{\beta q}{r}}u \bigg)^{\frac{r}{q}} ds &  \approx  \sum_{i=k+1}^{\infty} 2^{-i(1-\beta)} \bigg(\int_{x_k}^{x_i} {\mathcal W}^{\frac{\beta q}{r}}u \bigg)^{\frac{r}{q}} \notag \\
&\approx \sum_{i=k}^{\infty} 2^{-i(1-\beta)} \bigg( \int_{x_i}^{x_{i+1}} {\mathcal W}^{\frac{\beta q}{r}}u \bigg)^{\frac{r}{q}}.
\end{align}
Therefore, in view of \eqref{int-wu-sum},
\begin{align}
\mathcal{B}_1 & \approx  \sup_{N+1 \leq k} \bigg(\int_{x_k}^b {\mathcal W}(s)^{-\beta}  w(s) \bigg(\int_{x_k}^s {\mathcal W}^{\frac{\beta q}{r}} u \bigg)^{\frac{r}{q}} ds  \bigg)^{\frac{1}{r}} V_p(a, x_k)  \notag \\
& \leq \sup_{N+1 \leq k} \esup_{t \in (x_{k-1}\ x_k)} \bigg(\int_t^{b} {\mathcal W}(s)^{-\beta } w(s) \bigg(\int_t^{s} {\mathcal W}^{\frac{\beta q}{r}}u \bigg)^{\frac{r}{q}} ds \bigg)^{\frac{1}{r}} V_p(a, t)  
= C_1. \label{B1*<C1}
\end{align}
Thus, combining \eqref{A1<C1} with \eqref{B1*<C1}, we have that $A_1 + \mathcal{B}_1 \lesssim C_1$. 

Conversely, using \eqref{V-cut}, we have
\begin{align*}
C_1 & \approx  \sup_{N+1 \leq k}   \bigg(\int_{x_k}^{b} {\mathcal W}(s)^{-\beta }w(s) \bigg(\int_{x_k}^s {\mathcal W}^{\frac{\beta q}{r}}u \bigg)^{\frac{r}{q}} ds \bigg)^{\frac{1}{r}}  V_p(x_{k-1}, x_k) \\
&+\sup_{N+1 \leq k} 2^{-\frac{k(1-\beta)}{r}}\esup_{t \in (x_{k-1}, x_k)}  \bigg(\int_{t}^{x_k} {\mathcal W}^{\frac{\beta q}{r}} u \bigg)^{\frac{1}{q}}  V_p(x_{k-1}, t) \\
&+\sup_{N+1 \leq k} \esup_{t \in (x_{k-1}, x_k)}  \bigg(\int_{t}^{x_k} {\mathcal W}(s)^{-\beta }w(s) \bigg(\int_{t}^s {\mathcal W}^{\frac{\beta q}{r}}u \bigg)^{\frac{r}{q}} ds \bigg)^{\frac{1}{r}}  V_p(x_{k-1}, t)\\
	& +  \sup_{N+2 \leq k} \bigg(\int_{x_{k-1}}^b {\mathcal W}(s)^{-\beta}w(s) \bigg(\int_{x_{k-1}}^s {\mathcal W}^{\frac{\beta q}{r}}u \bigg)^{\frac{r}{q}} ds \bigg)^{\frac{1}{r}}  V_p(a, x_{k-1}).
\end{align*}
Then, using \eqref{int-wu-sum}, we arrive at
\begin{align} \label{C1<A1+B1*}
C_1 & \lesssim \sup_{N+1\leq k} 2^{-\frac{k(1-\beta)}{r}} \esup_{t \in (x_{k-1}, x_k)} \bigg(\int_t^{x_k} {\mathcal W}^{\frac{\beta q}{r}}u \bigg)^{\frac{1}{q}}  V_p(a, t) \notag\\
& \quad + \sup_{N+1 \leq k} \bigg(\sum_{i=k}^{\infty} 2^{-i(1-\beta)} \bigg(\int_{x_i}^{x_{i+1}}{\mathcal W}^{\frac{\beta q}{r}}u \bigg)^{\frac{r}{q}} \bigg)^{\frac{1}{r}}  V_p(a, x_{k}) \notag \\
	& \leq A_1 + \mathcal{B}_1.
\end{align}

As a result, we arrive at the conclusion that the best constant $C$ in \eqref{main-iterated} satisfies $C \approx C_1$.

\rm(ii) Let $r < p \leq q$. Then, we have from [Theorem~\ref{T:disc.char.}, (ii)] that the best constant in \eqref{main-iterated} satisfies $C\approx \mathcal{A}_2 + \mathcal{B}_2$. We will start by showing that $\mathcal{A}_2 + \mathcal{B}_2 \approx A_2 + B_2$, where
\begin{equation*}
A_2 := \bigg(\sum_{k=N+1}^{\infty} 2^{-k(1-\beta)\frac{p}{p-r}} \esup_{t \in (a, x_{k})} \bigg(\int_t^{x_{k}} {\mathcal W}^{\frac{\beta q}{r}}u\bigg)^{\frac{pr}{q(p-r)}} V_p(a, t)^{\frac{pr}{p-r}} \bigg)^{\frac{p-r}{pr}}
\end{equation*}
and
\begin{align} \label{B2}
B_2 &:= \bigg(\sum_{k=N+1}^{\infty} 2^{-k(1-\beta)} \bigg(\int_{x_k}^{x_{k+1}} {\mathcal W}^{\frac{\beta q}{r}}u \bigg)^{\frac{r}{q}} \notag \\ &\hspace{2cm} \times  \bigg(\sum_{i=k+2}^{\infty} 2^{-i(1-\beta)} \bigg(\int_{x_i}^{x_{i+1}} {\mathcal W}^{\frac{\beta q}{r}}u \bigg)^{\frac{r}{q}}  \bigg)^{\frac{r}{p-r}} 	V_p(a, x_k)^{\frac{pr}{p-r}} \bigg)^{\frac{p-r}{pr}}.
\end{align}
We have $\mathcal{A}_2 \leq A_2$ by the definitions of $\mathcal{A}_2$ and $A_2$. On the other hand, 
\begin{align}
\mathcal{B}_2 &\approx B_2  + \bigg(\sum_{k=N+1}^{\infty} 2^{-k(1-\beta)\frac{p}{p-r}} \bigg(\int_{x_k}^{x_{k+1}} {\mathcal W}^{\frac{\beta q}{r}} u \bigg)^{\frac{pr}{q(p-r)}}  V_p(a, x_k)^{\frac{pr}{p-r}} \bigg)^{\frac{p-r}{pr}} \notag \\
& \hspace{1cm} + 
\bigg(\sum_{k=N+1}^{\infty} 2^{-k(1-\beta)\frac{p}{p-r}} \bigg(\int_{x_{k}}^{x_{k+1}} {\mathcal W}^{\frac{\beta q}{r}}u \bigg)^{\frac{r}{q}}  \bigg(\int_{x_{k+1}}^{x_{k+2}} {\mathcal W}^{\frac{\beta q}{r}}u \bigg)^{\frac{r^2}{q(p-r)}}  V_p(a, x_k)^{\frac{pr}{p-r}} \bigg)^{\frac{p-r}{pr}} \notag\\
& \lesssim B_2 + \bigg(\sum_{k=N+1}^{\infty} 2^{-k(1-\beta)\frac{p}{p-r}} \bigg(\int_{x_k}^{x_{k+2}} {\mathcal W}^{\frac{\beta q}{r}}u \bigg)^{\frac{pr}{q(p-r)}}  V_p(a, x_k)^{\frac{pr}{p-r}} \bigg)^{\frac{p-r}{pr}} \notag \\
& \lesssim B_2 + \bigg(\sum_{k=N+1}^{\infty} 2^{-k(1-\beta)\frac{p}{p-r}} \esup_{t \in (a, x_k)} \bigg(\int_{t}^{x_{k+2}} {\mathcal W}^{\frac{\beta q}{r}} u \bigg)^{\frac{pr}{q(p-r)}}  V_p(a, t)^{\frac{pr}{p-r}} \bigg)^{\frac{p-r}{pr}} \notag \\
& \lesssim B_2 + A_2 \label{B2*<A2+B2}
\end{align}
holds, hence $\mathcal{A}_2 + \mathcal{B}_2 \lesssim A_2 + B_2$. 

Next, we will show that $A_2 \lesssim \mathcal{A}_2 + \mathcal{B}_2$. Observe that,
\begin{align*}
A_2 & = \bigg(\sum_{k=N+1}^{\infty} 2^{-k(1-\beta)\frac{p}{p-r}}  \sup_{N+1 \leq i \leq k} \esup_{t \in (x_{i-1}, x_{i})} \bigg(\int_t^{x_{k}} {\mathcal W}^{\frac{\beta q}{r}} u\bigg)^{\frac{pr}{q(p-r)}} V_p(a, t)^{\frac{pr}{p-r}} \bigg)^{\frac{p-r}{pr}}  \\
& \approx \bigg(\sum_{k=N+1}^{\infty} 2^{-k(1-\beta)\frac{p}{p-r}}  \sup_{N+1 \leq i \leq k} \esup_{t \in (x_{i-1}, x_{i})} \bigg(\int_t^{x_{i}} {\mathcal W}^{\frac{\beta q}{r}}u\bigg)^{\frac{pr}{q(p-r)}} V_p(a, t)^{\frac{pr}{p-r}} \bigg)^{\frac{p-r}{pr}}  \\
& \quad + \bigg(\sum_{k=N+2}^{\infty} 2^{-k(1-\beta)\frac{p}{p-r}}  \sup_{N+1 \leq i < k} \bigg(\int_{x_i}^{x_{k}} {\mathcal W}^{\frac{\beta q}{r}}u\bigg)^{\frac{pr}{q(p-r)}} V_p(a, x_i)^{\frac{pr}{p-r}} \bigg)^{\frac{p-r}{pr}}.
\end{align*}
Applying \eqref{sum-sup} for the first term and \eqref{3-sum-equiv} for the second term, we obtain that
\begin{align*}
A_2 &  \approx \bigg(\sum_{k=N+1}^{\infty} 2^{-k(1-\beta)\frac{p}{p-r}} \esup_{t \in (x_{k-1}, x_{k})} \bigg(\int_t^{x_{k}} {\mathcal W}^{\frac{\beta q}{r}}u\bigg)^{\frac{pr}{q(p-r)}} V_p(a, t)^{\frac{pr}{p-r}} \bigg)^{\frac{p-r}{pr}}  \\
& \quad + \bigg(\sum_{k=N+2}^{\infty} 2^{-k(1-\beta)\frac{p}{p-r}}  \bigg(\int_{x_{k-1}}^{x_{k}} {\mathcal W}^{\frac{\beta q}{r}} u\bigg)^{\frac{pr}{q(p-r)}} V_p(a, x_{k-1})^{\frac{pr}{p-r}} \bigg)^{\frac{p-r}{pr}}.
\end{align*}
Using
\eqref{V-cut}
we arrive at
\begin{align*}
A_2 & \lesssim \mathcal{A}_2 + \bigg(\sum_{k=N+1}^{\infty} 2^{-k(1-\beta)\frac{p}{p-r}} \bigg(\int_{x_{k}}^{x_{k+1}} {\mathcal W}^{\frac{\beta q}{r}} u\bigg)^{\frac{pr}{q(p-r)}} V_p(a, x_{k})^{\frac{pr}{p-r}} \bigg)^{\frac{p-r}{pr}} \lesssim \mathcal{A}_2 + \mathcal{B}_2.
\end{align*}
Furthermore, it is clear from the definitions of $\mathcal{B}_2$ and $B_2$ that $B_2 \leq \mathcal{B}_2$. Then, we have $A_2 + B_2 \lesssim \mathcal{A}_2 + \mathcal{B}_2$, as well. Consequently, $C \approx A_2 + B_2$ holds.  

Next, we will prove that $A_2 + B_2 \approx C_2 + C_3$.
First of all, applying \eqref{int.equiv} with $\alpha = \frac{p(1-\beta)}{p-r} $ and
\begin{equation*}
    h(x) = \esup_{t \in (a, x)} \bigg(\int_t^{x} {\mathcal W}^{\frac{\beta q}{r}}u \bigg)^{\frac{pr}{q(p-r)}}  V_p(a, t)^{\frac{pr}{p-r}}, \quad x\in (a, b),
\end{equation*}
it is clear that
\begin{equation}\label{A2=C2}
A_2 \approx \bigg(\int_a^b \mathcal{W}(x)^{\frac{p(1-\beta)}{p-r}-1 } w(x) \esup_{t \in (a, x)} \bigg(\int_t^{x}  {\mathcal W}^{\frac{\beta q}{r}}u \bigg)^{\frac{pr}{q(p-r)}}  V_p(a, t)^{\frac{pr}{p-r}} dx \bigg)^{\frac{p-r}{pr}} = C_2.
\end{equation}

On the other hand, using \eqref{int-wu-sum},
\begin{align}
B_2 & \approx \bigg(\sum_{k=N+1}^{\infty} 2^{-k(1-\beta)}  \bigg(\int_{x_{k+2}}^b {\mathcal W}(s)^{-\beta } w(s) \bigg(\int_{x_{k+2}}^s {\mathcal W}^{\frac{\beta q}{r}} u \bigg)^{\frac{r}{q}} ds  \bigg)^{\frac{r}{p-r}} \notag\\
& 
\hspace{5cm} \times\bigg(\int_{x_k}^{x_{k+1}} {\mathcal W}^{\frac{\beta q}{r}}u \bigg)^{\frac{r}{q}}  V_p(a, x_k)^{\frac{pr}{p-r}} \bigg)^{\frac{p-r}{pr}}\notag\\
& \leq \bigg(\sum_{k=N+1}^{\infty} 2^{-k(1-\beta)}  \bigg(\int_{x_{k+2}}^b {\mathcal W}(s)^{-\beta }w(s) \bigg(\int_{x_{k+2}}^s {\mathcal W}^{\frac{\beta q}{r}}u \bigg)^{\frac{r}{q}} ds  \bigg)^{\frac{r}{p-r}} \notag\\
&\hspace{5cm} \times \esup_{t \in (a, x_{k+1})} \bigg(\int_{t}^{x_{k+1}}{\mathcal W}^{\frac{\beta q}{r}} u \bigg)^{\frac{r}{q}}  V_p(a, t)^{\frac{pr}{p-r}} \bigg)^{\frac{p-r}{pr}}\notag\\
& \approx\bigg(\sum_{k=N+1}^{\infty} \int_{x_{k+1}}^{x_{k+2}} {\mathcal W}(x)^{-\beta}w(x) dx  \bigg(\int_{x_{k+2}}^b {\mathcal W}(s)^{-\beta} w(s) \bigg(\int_{x_{k+2}}^s {\mathcal W}^{\frac{\beta q}{r}} u \bigg)^{\frac{r}{q}} ds  \bigg)^{\frac{r}{p-r}} \notag\\
& \hspace{5cm} \times \esup_{t \in (a, x_{k+1})} \bigg(\int_{t}^{x_{k+1}} {\mathcal W}^{\frac{\beta q}{r}}u \bigg)^{\frac{r}{q}}  V_p(a, t)^{\frac{pr}{p-r}} \bigg)^{\frac{p-r}{pr}}\notag\\
& \leq \bigg(\sum_{k=N+1}^{\infty} \int_{x_{k+1}}^{x_{k+2}} {\mathcal W}(x)^{-\beta}w(x)  \bigg(\int_{x}^b {\mathcal W}(s)^{-\beta}w(s) \bigg(\int_{x}^s {\mathcal W}^{\frac{\beta q}{r}}u \bigg)^{\frac{r}{q}} ds  \bigg)^{\frac{r}{p-r}} \notag\\
&\hspace{5cm} \times \esup_{t \in (a, x)} \bigg(\int_{t}^{x} {\mathcal W}^{\frac{\beta q}{r}}u \bigg)^{\frac{r}{q}}  V_p(a, t)^{\frac{pr}{p-r}} dx \bigg)^{\frac{p-r}{pr}} \notag\\
&\leq C_3. \label{B2<C3}
\end{align}
Combination of \eqref{A2=C2} and \eqref{B2<C3} yield that $A_2 + B_2 \lesssim C_2 + C_3$. 

Conversely,
\begin{align*}
C_3  &= \bigg( \sum_{k=N+1}^{\infty} \int_{x_{k-1}}^{x_k}  \bigg(\int_{x}^b {\mathcal W}(s)^{-\beta}w(s) \bigg(\int_{x}^s {\mathcal W}^{\frac{\beta q}{r}}u \bigg)^{\frac{r}{q}} ds  \bigg)^{\frac{r}{p-r}} {\mathcal W}(x)^{-\beta}w(x) \\
&\hskip+2cm \times\esup_{t \in (a, x)} \bigg(\int_{t}^{x} {\mathcal W}^{\frac{\beta q}{r}} u \bigg)^{\frac{r}{q}} V_p(a, t)^{\frac{pr}{p-r}}  dx \bigg)^{\frac{p-r}{pr}} \\
& \approx \bigg( \sum_{k=N+1}^{\infty} \int_{x_{k-1}}^{x_k}  \bigg(\int_{x}^{x_k} {\mathcal W}(s)^{-\beta} w(s) \bigg(\int_{x}^s {\mathcal W}^{\frac{\beta q}{r}} u \bigg)^{\frac{r}{q}} ds  \bigg)^{\frac{r}{p-r}} {\mathcal W}(x)^{-\beta}w(x) \\ 
&\hskip+2cm \times \esup_{t \in (a, x)} \bigg(\int_{t}^{x} {\mathcal W}^{\frac{\beta q}{r}} u \bigg)^{\frac{r}{q}} V_p(a, t)^{\frac{pr}{p-r}} dx \bigg)^{\frac{p-r}{pr}} \\
& \hspace{0.2cm} + \bigg( \sum_{k=N+1}^{\infty} \int_{x_{k-1}}^{x_k}  \bigg(\int_{x_k}^b {\mathcal W}(s)^{-\beta}w(s) \bigg(\int_{x}^s {\mathcal W}^{\frac{\beta q}{r}} u \bigg)^{\frac{r}{q}} ds  \bigg)^{\frac{r}{p-r}} {\mathcal W}(x)^{-\beta} w(x) \\
&\hskip+2cm\times \esup_{t \in (a, x)} \bigg(\int_{t}^{x} {\mathcal W}^{\frac{\beta q}{r}} u \bigg)^{\frac{r}{q}} V_p(a, t)^{\frac{pr}{p-r}}  dx \bigg)^{\frac{p-r}{pr}} \\
& \approx \bigg( \sum_{k=N+1}^{\infty} \int_{x_{k-1}}^{x_k}  \bigg(\int_{x}^{x_k} {\mathcal W}(s)^{-\beta} w(s) \bigg(\int_{x}^s {\mathcal W}^{\frac{\beta q}{r}} u \bigg)^{\frac{r}{q}} ds  \bigg)^{\frac{r}{p-r}} {\mathcal W}(x)^{-\beta}w(x) \\ 
&\hskip+2cm \times \esup_{t \in (a, x)} \bigg(\int_{t}^{x} {\mathcal W}^{\frac{\beta q}{r}} u \bigg)^{\frac{r}{q}} V_p(a, t)^{\frac{pr}{p-r}} dx \bigg)^{\frac{p-r}{pr}} \\
& \hspace{0.2cm} + \bigg( \sum_{k=N+1}^{\infty} \bigg(\int_{x_k}^b {\mathcal W}(s)^{-\beta}w(s)  ds  \bigg)^{\frac{r}{p-r}} \int_{x_{k-1}}^{x_k}  \bigg(\int_{x}^{x_k} {\mathcal W}^{\frac{\beta q}{r}} u \bigg)^{\frac{r^2}{q(p-r)}}  {\mathcal W}(x)^{-\beta} w(x) \\
&\hskip+2cm\times \esup_{t \in (a, x)} \bigg(\int_{t}^{x} {\mathcal W}^{\frac{\beta q}{r}} u \bigg)^{\frac{r}{q}} V_p(a, t)^{\frac{pr}{p-r}}  dx \bigg)^{\frac{p-r}{pr}} \\
& \hspace{0.2cm} + \bigg( \sum_{k=N+1}^{\infty} \int_{x_{k-1}}^{x_k}  \bigg(\int_{x_k}^b {\mathcal W}(s)^{-\beta}w(s) \bigg(\int_{x_k}^s {\mathcal W}^{\frac{\beta q}{r}} u \bigg)^{\frac{r}{q}} ds  \bigg)^{\frac{r}{p-r}} {\mathcal W}(x)^{-\beta} w(x) \\
&\hskip+2cm\times \esup_{t \in (a, x)} \bigg(\int_{t}^{x} {\mathcal W}^{\frac{\beta q}{r}} u \bigg)^{\frac{r}{q}} V_p(a, t)^{\frac{pr}{p-r}}  dx \bigg)^{\frac{p-r}{pr}}.
\end{align*}
Since 
\begin{equation*}
\bigg(\int_{x_k}^b {\mathcal W}(s)^{-\beta}w(s)  ds  \bigg)^{\frac{r}{p-r}} \approx 2^{-k(1-\beta)\frac{r}{p-r}},
\end{equation*}
\begin{align*}
C_3  
& \approx \bigg( \sum_{k=N+1}^{\infty} \int_{x_{k-1}}^{x_k}  \bigg(\int_{x}^{x_k} {\mathcal W}(s)^{-\beta} w(s) \bigg(\int_{x}^s {\mathcal W}^{\frac{\beta q}{r}} u \bigg)^{\frac{r}{q}} ds  \bigg)^{\frac{r}{p-r}} {\mathcal W}(x)^{-\beta} w(x)  \\
&\hskip+2cm \times  \esup_{t \in (a, x)} \bigg(\int_{t}^{x} {\mathcal W}^{\frac{\beta q}{r}} u \bigg)^{\frac{r}{q}}  V_p(a, t)^{\frac{pr}{p-r}}  dx \bigg)^{\frac{p-r}{pr}} \\
& \hspace{0.2cm} + \bigg( \sum_{k=N+1}^{\infty} 2^{-k\frac{(1-\beta)r}{p-r}} \int_{x_{k-1}}^{x_k} \bigg(\int_{x}^{x_k} {\mathcal W}^{\frac{\beta q}{r}} u \bigg)^{\frac{r^2}{q(p-r)}} \mathcal{W}(x)^{-\beta}  w(x) \\
&\hspace{2cm} \times\esup_{t \in (a, x)} \bigg(\int_{t}^{x} {\mathcal W}^{\frac{\beta q}{r}} u \bigg)^{\frac{r}{q}}  V_p(a, t)^{\frac{pr}{p-r}}  dx \bigg)^{\frac{p-r}{pr}} \\
& \hspace{0.2cm}  + \bigg( \sum_{k=N+1}^{\infty} \bigg(\int_{x_k}^b {\mathcal W}(s)^{-\beta}w(s) \bigg(\int_{x_k}^s {\mathcal W}^{\frac{\beta q}{r}} u \bigg)^{\frac{r}{q}} ds  \bigg)^{\frac{r}{p-r}}\\
&\hskip+2cm \times \int_{x_{k-1}}^{x_k}{\mathcal W}(x)^{-\beta} w(x) \esup_{t \in (a, x)} \bigg(\int_{t}^{x} {\mathcal W}^{\frac{\beta q}{r}} u \bigg)^{\frac{r}{q}}  V_p(a, t)^{\frac{pr}{p-r}}  dx \bigg)^{\frac{p-r}{pr}} \\
& =: C_{3,1} + C_{3,2} + C_{3,3}.
\end{align*}
Note that
\begin{align}\label{w-calc_1}
\int_{x_{k-1}}^{x_k}  \bigg(\int_{x}^{x_k}{\mathcal W}^{-\beta} w \bigg)^{\frac{r}{p-r}}{\mathcal W}(x)^{-\beta} w(x) dx \approx \bigg(\int_{x_{k-1}}^{x_k} \mathcal{W}^{-\beta} w\bigg)^{\frac{p}{p-r}} \approx 2^{-k(1-\beta)\frac{p}{p-r}}.
\end{align}
In view of \eqref{w-calc_1}, we get
\begin{align*}
C_{3,1} & \leq \bigg( \sum_{k=N+1}^{\infty} \int_{x_{k-1}}^{x_k}  \bigg(\int_{x}^{x_k}{\mathcal W}^{-\beta} w \bigg)^{\frac{r}{p-r}} {\mathcal W}(x)^{-\beta}w(x) \\
& \hspace{2cm} \times \esup_{t \in (a, x)} \bigg(\int_{t}^{x_k} {\mathcal W}^{\frac{\beta q}{r}} u \bigg)^{\frac{pr}{q(p-r)}}  V_p(a, t)^{\frac{pr}{p-r}} dx \bigg)^{\frac{p-r}{pr}} \\
& \leq \bigg( \sum_{k=N+1}^{\infty}  \int_{x_{k-1}}^{x_k}  \bigg(\int_{x}^{x_k}{\mathcal W}^{-\beta} w \bigg)^{\frac{r}{p-r}}{\mathcal W}(x)^{-\beta} w(x) dx \\
& \hspace{2cm} \times \esup_{t \in (a, x_k)} \bigg(\int_{t}^{x_k} {\mathcal W}^{\frac{\beta q}{r}} u \bigg)^{\frac{pr}{q(p-r)}}  V_p(a, t)^{\frac{pr}{p-r}} \bigg)^{\frac{p-r}{pr}} \\
& \approx A_2
\end{align*}
and
\begin{align*}
C_{3,2}&\leq \bigg( \sum_{k=N+1}^{\infty} 2^{-k\frac{(1-\beta)r}{p-r}} \int_{x_{k-1}}^{x_k} \mathcal{W}(x)^{-\beta} w(x) \\
& \hskip+4cm \times \esup_{t \in (a, x)} \bigg(\int_{t}^{x_k} {\mathcal W}^{\frac{\beta q}{r}} u \bigg)^{\frac{r^2}{q(p-r)}} \bigg(\int_{t}^{x}{\mathcal W}^{\frac{\beta q}{r}} u \bigg)^{\frac{r}{q}}  V_p(a, t)^{\frac{pr}{p-r}} dx \bigg)^{\frac{p-r}{pr}}\\
&\leq \bigg( \sum_{k=N+1}^{\infty} 2^{-k\frac{(1-\beta)r}{p-r}} \int_{x_{k-1}}^{x_k}\mathcal{W}(x)^{-\beta} w(x)  \esup_{t \in (a, x)} \bigg(\int_{t}^{x_k} {\mathcal W}^{\frac{\beta q}{r}} u \bigg)^{\frac{pr}{q(p-r)}}  V_p(a, t)^{\frac{pr}{p-r}} dx \bigg)^{\frac{p-r}{pr}}\\
&\leq \bigg( \sum_{k=N+1}^{\infty} 2^{-k\frac{(1-\beta)r}{p-r}} \int_{x_{k-1}}^{x_k} \mathcal{W}(x)^{-\beta} w(x) dx  \esup_{t \in (a, x_k)} \bigg(\int_{t}^{x_k} {\mathcal W}^{\frac{\beta q}{r}} u \bigg)^{\frac{pr}{q(p-r)}}  V_p(a, t)^{\frac{pr}{p-r}}  \bigg)^{\frac{p-r}{pr}}\\
&\approx \bigg( \sum_{k=N+1}^{\infty} 2^{-k\frac{(1-\beta)p}{p-r}} \  \esup_{t \in (a, x_k)} \bigg(\int_{t}^{x_k} {\mathcal W}^{\frac{\beta q}{r}} u \bigg)^{\frac{pr}{q(p-r)}}  V_p(a, t)^{\frac{pr}{p-r}}  \bigg)^{\frac{p-r}{pr}}\\
& = A_2.
\end{align*}
Furthermore, using \eqref{w-calc_1} once more,
\begin{align*}
C_{3,3} & \lesssim \bigg( \sum_{k=N+1}^{\infty} 2^{-k(1-\beta)} \bigg(\int_{x_k}^b {\mathcal W}(s)^{-\beta}w(s) \bigg(\int_{x_k}^s {\mathcal W}^{\frac{\beta q}{r}} u \bigg)^{\frac{r}{q}} ds  \bigg)^{\frac{r}{p-r}}  \\
&\hspace{4cm} \times\esup_{t \in (a, x_k)} \bigg(\int_{t}^{x_k} {\mathcal W}^{\frac{\beta q}{r}} u \bigg)^{\frac{r}{q}}  V_p(a, t)^{\frac{pr}{p-r}}  \bigg)^{\frac{p-r}{pr}}\\
& = \bigg( \sum_{k=N+1}^{\infty} 2^{-k(1-\beta)} \bigg(\int_{x_k}^b{\mathcal W}(s)^{-\beta} w(s) \bigg(\int_{x_k}^s {\mathcal W}^{\frac{\beta q}{r}} u \bigg)^{\frac{r}{q}} ds  \bigg)^{\frac{r}{p-r}} \\
& \hspace{4cm} \times  \sup_{N+1 \leq i \leq  k} \esup_{t \in (x_{i-1}, x_i)} \bigg(\int_{t}^{x_k} {\mathcal W}^{\frac{\beta q}{r}} u \bigg)^{\frac{r}{q}}  V_p(a, t)^{\frac{pr}{p-r}}  \bigg)^{\frac{p-r}{pr}}\\
& \approx \bigg( \sum_{k=N+1}^{\infty} 2^{-k(1-\beta)} \bigg(\int_{x_k}^b {\mathcal W}(s)^{-\beta}w(s) \bigg(\int_{x_k}^s {\mathcal W}^{\frac{\beta q}{r}} u \bigg)^{\frac{r}{q}} ds  \bigg)^{\frac{r}{p-r}} \\
& \hspace{4cm} \times \sup_{N+1 \leq i \leq k} \esup_{t \in (x_{i-1}, x_i)} \bigg(\int_{t}^{x_i} {\mathcal W}^{\frac{\beta q}{r}} u \bigg)^{\frac{r}{q}} V_p(a, t)^{\frac{pr}{p-r}}  \bigg)^{\frac{p-r}{pr}} \\
& + \bigg( \sum_{k=N+2}^{\infty} 2^{-k(1-\beta)} \bigg(\int_{x_k}^b {\mathcal W}(s)^{-\beta}w(s) \bigg(\int_{x_k}^s {\mathcal W}^{\frac{\beta q}{r}} u \bigg)^{\frac{r}{q}} ds  \bigg)^{\frac{r}{p-r}}  \\
& \hspace{4cm} \times \sup_{N+1 \leq i <  k} \bigg(\int_{x_i}^{x_{k}} {\mathcal W}^{\frac{\beta q}{r}} u \bigg)^{\frac{r}{q}} V_p(a, x_i)^{\frac{pr}{p-r}} \bigg)^{\frac{p-r}{pr}}  \\
& =: I + II.
\end{align*}
Since, the sequence $\{a_k\}_{k=N+1}^{\infty}$, with 
$$
a_k =: 2^{-k(1-\beta)} \bigg(\int_{x_k}^b {\mathcal W}(s)^{-\beta} w(s) \bigg(\int_{x_k}^s {\mathcal W}^{\frac{\beta q}{r}} u \bigg)^{\frac{r}{q}} ds  \bigg)^{\frac{r}{p-r}} 
$$
is geometrically decreasing, \eqref{sum-sup} yields that
\begin{align*}
I & \approx \bigg( \sum_{k=N+1}^{\infty} 2^{-k(1-\beta)}  \bigg(\int_{x_k}^b {\mathcal W}(s)^{-\beta}w(s) \bigg(\int_{x_k}^s {\mathcal W}^{\frac{\beta q}{r}} u \bigg)^{\frac{r}{q}} ds  \bigg)^{\frac{r}{p-r}} \\
& \hspace{4cm} \times\esup_{t \in (x_{k-1}, x_k)} \bigg(\int_{t}^{x_k} {\mathcal W}^{\frac{\beta q}{r}} u \bigg)^{\frac{r}{q}}  V_p(a, t)^{\frac{pr}{p-r}}  \bigg)^{\frac{p-r}{pr}}.
\end{align*}
Let  $y_k \in [x_{k-1}, x_k]$, $N+1 \leq k$ be such that
\begin{align*}
\esup_{t \in (x_{k-1}, x_k)} \bigg(\int_{t}^{x_k} {\mathcal W}^{\frac{\beta q}{r}} u \bigg)^{\frac{r}{q}}  V_p(a, t)^{\frac{pr}{p-r}}   \lesssim \bigg(\int_{y_k}^{x_k} {\mathcal W}^{\frac{\beta q}{r}} u \bigg)^{\frac{r}{q}}  V_p(a, y_k)^{\frac{pr}{p-r}}. 
\end{align*}
Then, we have
\begin{align*}
I & \lesssim \bigg( \sum_{k=N+1}^{\infty} 2^{-k(1-\beta)} \bigg(\int_{x_k}^b {\mathcal W}(s)^{-\beta}w(s) \bigg(\int_{x_k}^s {\mathcal W}^{\frac{\beta q}{r}} u \bigg)^{\frac{r}{q}} ds  \bigg)^{\frac{r}{p-r}} \\
&\hspace{4cm} \times \bigg(\int_{y_k}^{x_k} {\mathcal W}^{\frac{\beta q}{r}} u \bigg)^{\frac{r}{q}}  V_p(a, y_k)^{\frac{pr}{p-r}} \bigg)^{\frac{p-r}{pr}}.
\end{align*}
Thus, \eqref{int-wu-sum} ensures that
\begin{align*}
I & \lesssim \bigg( \sum_{k=N+1}^{\infty} 2^{-k(1-\beta)} \bigg(\sum_{i=k}^{\infty} 2^{-i(1-\beta)} \bigg( \int_{x_i}^{x_{i+1}} {\mathcal W}^{\frac{\beta q}{r}} u \bigg)^{\frac{r}{q}}  \bigg)^{\frac{r}{p-r}} \\
&\hspace{4cm} \times \bigg(\int_{y_k}^{x_k} {\mathcal W}^{\frac{\beta q}{r}} u \bigg)^{\frac{r}{q}}  V_p(a, y_k)^{\frac{pr}{p-r}} \bigg)^{\frac{p-r}{pr}}\\
& \leq
\bigg( \sum_{k=N+1}^{\infty} 2^{-k(1-\beta)} \bigg(\sum_{i=k}^{\infty} 2^{-i(1-\beta)} \bigg( \int_{y_i}^{y_{i+2}} {\mathcal W}^{\frac{\beta q}{r}} u \bigg)^{\frac{r}{q}}  \bigg)^{\frac{r}{p-r}} \\
&\hspace{4cm} \times \bigg(\int_{y_k}^{y_{k+1}} {\mathcal W}^{\frac{\beta q}{r}} u \bigg)^{\frac{r}{q}}  V_p(a, y_k)^{\frac{pr}{p-r}} \bigg)^{\frac{p-r}{pr}}.
\end{align*}
Moreover,
\begin{equation*}
2^{-k} \approx \int_{x_k}^b w \leq \int_{y_k}^b w \leq \int_{x_{k-1}}^b w \approx 2^{-(k-1)}, \quad N+1 \leq k. 
\end{equation*}
As a result, $\{y_k\}_{k=N+1}^{\infty}$ is a discretizing sequence of ${\mathcal W}$, as well. This fact together with \eqref{B2*<A2+B2} yield $I \lesssim \mathcal{B}_2 \lesssim A_2 +B_2$. 

On the other hand, applying \eqref{3-sum-equiv} with
$$
\tau_k =  2^{-k(1-\beta)} \bigg(\int_{x_k}^b{\mathcal W}(s)^{-\beta} w(s) \bigg(\int_{x_k}^s {\mathcal W}^{\frac{\beta q}{r}} u \bigg)^{\frac{r}{q}} ds  \bigg)^{\frac{r}{p-r}}, \quad \sigma_k = V_p(a, x_k)^{\frac{pr}{p-r}}, \quad \alpha=\frac{r}{q},
$$
gives
\begin{align*}
II &\approx  \bigg( \sum_{k=N+2}^{\infty} 2^{-k(1-\beta)} \bigg(\int_{x_k}^b {\mathcal W}(s)^{-\beta}w(s) \bigg(\int_{x_k}^s {\mathcal W}^{\frac{\beta q}{r}} u \bigg)^{\frac{r}{q}} ds  \bigg)^{\frac{r}{p-r}}  \\
&\hspace{4cm} \times \bigg(\int_{x_{k-1}}^{x_{k}} {\mathcal W}^{\frac{\beta q}{r}} u \bigg)^{\frac{r}{q}} V_p(a, x_{k-1})^{\frac{pr}{p-r}} \bigg)^{\frac{p-r}{pr}}.  
\end{align*}
Lastly, using \eqref{int-wu-sum} and \eqref{B2*<A2+B2}, we obtain that
\begin{align*}
II &  \approx  \bigg( \sum_{k=N+2}^{\infty} 2^{-k(1-\beta)} \bigg(\sum_{i=k}^{\infty} 2^{-i(1-\beta)} \bigg( \int_{x_i}^{x_{i+1}} {\mathcal W}^{\frac{\beta q}{r}} u \bigg)^{\frac{r}{q}} \bigg)^{\frac{r}{1-r}} \bigg(\int_{x_{k-1}}^{x_k} {\mathcal W}^{\frac{\beta q}{r}} u \bigg)^{\frac{r}{q}} V_p(a, x_{k-1})^{\frac{pr}{p-r}}  \bigg)^{\frac{p-r}{pr}}\\
&\leq A_2 + B_2.
\end{align*}
Therefore, 
\begin{equation}\label{C3<A2+B2}
C_3 \lesssim A_2 + B_2.
\end{equation}

Finally, combination of \eqref{A2=C2} and \eqref{C3<A2+B2} yield, $C_2 + C_3 \approx  A_2 + B_2$. Accordingly, the best constant $C$ in \eqref{main-iterated} satisfies $C \approx C_2 +C_3 $.

\rm(iii) Let $ q< p \le r$. According to [Theorem~\ref{T:disc.char.}, (iii)], the best constant in \eqref{main-iterated} satisfies $C\approx \mathcal{A}_3 + \mathcal{B}_1$. We will begin our proof by showing that $\mathcal{A}_3 + \mathcal{B}_1 \approx A_3 + \mathcal{B}_1$, where
\begin{equation*}
A_3 := \sup_{N+1 \leq k <\infty} 2^{-\frac{k(1-\beta)}{r}} \bigg(\int_a^{x_{k}} \bigg(\int_t^{x_{k}} {\mathcal W}^{\frac{\beta q}{r}} u \bigg)^{\frac{q}{p-q}} {\mathcal W}(t)^{\frac{\beta q}{r}} u(t) \,  V_p(a,t)^{\frac{pq}{p-q}}dt \bigg)^{\frac{p-q}{pq}}.
\end{equation*}
It is clear that $\mathcal{A}_3 \leq A_3$, the proof of this part is complete if we show that $A_3 \lesssim \mathcal{A}_3 + \mathcal{B}_1$. Assume that $\max\{\mathcal{A}_3, \mathcal{B}_1\}<\infty$. Then,
\begin{equation*}
\bigg(\int_{x_{k-1}}^{x_{k}} \bigg(\int_t^{x_{k}} {\mathcal W}^{\frac{\beta q}{r}} u \bigg)^{\frac{q}{p-q}} {\mathcal W}(t)^{\frac{\beta q}{r}} u(t) \, V_p(a,t)^{\frac{pq}{p-q}}dt \bigg)^{\frac{p-q}{pq}} <\infty, \quad k\geq N+1,
\end{equation*}
holds. Also, for each $t\in [x_{k-1}, x_k]$, $k\geq N+1$, we have
\begin{equation*}
\bigg(\int_t^{x_{k}} {\mathcal W}^{\frac{\beta q}{r}} u \bigg)^{\frac{1}{q}} V_p(a,t) \lesssim \bigg(\int_{t}^{x_{k}} \bigg(\int_s^{x_{k}} {\mathcal W}^{\frac{\beta q}{r}} u \bigg)^{\frac{q}{p-q}} {\mathcal W}(s)^{\frac{\beta q}{r}} u(s) \, V_p(a,s)^{\frac{pq}{p-q}} ds \bigg)^{\frac{p-q}{pq}}.
\end{equation*}
Therefore, we have 
\begin{equation}\label{lim-0}
\lim_{t \rightarrow x_k-}    \bigg(\int_t^{x_{k}} {\mathcal W}^{\frac{\beta q}{r}} u \bigg)^{\frac{1}{q}} V_p(a,t) = 0. 
\end{equation}
In that case, integration by parts gives
\begin{align}
\bigg(\int_a^{x_{k}} & \bigg(\int_t^{x_{k}} {\mathcal W}^{\frac{\beta q}{r}} u \bigg)^{\frac{q}{p-q}} {\mathcal W}(t)^{\frac{\beta q}{r}} u(t) \, V_p(a,t)^{\frac{pq}{p-q}}dt \bigg)^{\frac{p-q}{pq}} \notag \\
& \approx \bigg(\int_a^{x_{k}} \bigg(\int_t^{x_{k}} {\mathcal W}^{\frac{\beta q}{r}} u \bigg)^{\frac{p}{p-q}}  d\big[ V_p(a,t)^{\frac{pq}{p-q}} \big] \bigg)^{\frac{p-q}{pq}} + \lim_{t \rightarrow a+} \bigg(\int_t^{x_{k}} {\mathcal W}^{\frac{\beta q}{r}} u \bigg)^{\frac{1}{q}}  V_p(a,t)\notag \\
& \approx \bigg(\sum_{i=N+1}^k \int_{x_{i-1}}^{x_{i}} \bigg(\int_t^{x_{i}} {\mathcal W}^{\frac{\beta q}{r}} u \bigg)^{\frac{p}{p-q}}  d\big[V_p(a,t)^{\frac{pq}{p-q}}\big]   \bigg)^{\frac{p-q}{pq}} \notag\\
& \quad +  \bigg(\sum_{i=N+1}^{k-1} \bigg(\int_{x_{i}}^{x_k} {\mathcal W}^{\frac{\beta q}{r}} u \bigg)^{\frac{p}{p-q}} \int_{x_{i-1}}^{x_{i}} d\big[V_p(a,t)^{\frac{pq}{p-q}}\big]   \bigg)^{\frac{p-q}{pq}} \notag
\\
& \qquad +  \lim_{t \rightarrow a+} \bigg(\int_t^{x_{k}} {\mathcal W}^{\frac{\beta q}{r}} u \bigg)^{\frac{1}{q}}  V_p(a,t).\notag
\end{align}
Moreover,  Minkowski's inequality with $\frac{p}{p-q} > 1$ yields that 
\begin{align*}
\bigg(\sum_{i=N+1}^{k-1} & \bigg(\int_{x_{i}}^{x_k} {\mathcal W}^{\frac{\beta q}{r}} u \bigg)^{\frac{p}{p-q}} \int_{x_{i-1}}^{x_{i}} d\big[V_p(a,t)^{\frac{pq}{p-q}}\big]   \bigg)^{\frac{p-q}{pq}} \\
& = \bigg(\sum_{i=N+1}^{k-1} \bigg(\sum_{j=i}^{k-1} \int_{x_{j}}^{x_{j+1}} {\mathcal W}^{\frac{\beta q}{r}} u \bigg)^{\frac{p}{p-q}} \int_{x_{i-1}}^{x_{i}} d\big[V_p(a,t)^{\frac{pq}{p-q}}\big]   \bigg)^{\frac{p-q}{pq}}\\
& \leq  \bigg(\sum_{j=N+1}^{k-1} \bigg(\int_{x_{j}}^{x_{j+1}} {\mathcal W}^{\frac{\beta q}{r}} u \bigg) \bigg(\sum_{i=N+1}^{j} \int_{x_{i-1}}^{x_i}  d\big[V_p(a,t)^{\frac{pq}{p-q}}\big]   \bigg)^{\frac{p-q}{p}} \bigg)^{\frac{1}{q}}\\
& =  \bigg(\sum_{j=N+1}^{k-1} \bigg(\int_{x_{j}}^{x_{j+1}} {\mathcal W}^{\frac{\beta q}{r}} u \bigg) \bigg( \int_{a}^{x_{j}}  d\big[V_p(a,t)^{\frac{pq}{p-q}}\big]   \bigg)^{\frac{p-q}{p}} \bigg)^{\frac{1}{q}}.
\end{align*}
Then, we arrive at
\begin{align}\label{IBP-estimate}
\bigg(\int_a^{x_{k}} & \bigg(\int_t^{x_{k}} {\mathcal W}^{\frac{\beta q}{r}} u \bigg)^{\frac{q}{p-q}} {\mathcal W}(t)^{\frac{\beta q}{r}} u(t) \, V_p(a,t)^{\frac{pq}{p-q}}dt \bigg)^{\frac{p-q}{pq}} \notag \\
& \lesssim \bigg(\sum_{i=N+1}^k \int_{x_{i-1}}^{x_{i}} \bigg(\int_t^{x_{i}} {\mathcal W}^{\frac{\beta q}{r}} u \bigg)^{\frac{p}{p-q}}  d\big[V_p(a,t)^{\frac{pq}{p-q}}\big]   \bigg)^{\frac{p-q}{pq}} \notag\\
& \quad +  \bigg(\sum_{j=N+1}^{k-1} \bigg(\int_{x_{j}}^{x_{j+1}} {\mathcal W}^{\frac{\beta q}{r}} u \bigg) \bigg(\int_{a}^{x_j}  d\big[V_p(a,t)^{\frac{pq}{p-q}}\big]   \bigg)^{\frac{p-q}{p}} \bigg)^{\frac{1}{q}} \notag\\
& \qquad +  \lim_{t \rightarrow a+} \bigg(\int_t^{x_{k}} {\mathcal W}^{\frac{\beta q}{r}} u \bigg)^{\frac{1}{q}}  V_p(a,t).
\end{align}
Now, we are in a position to find the upper estimate for $A_3$. Using  \eqref{IBP-estimate}, we have that
\begin{align*}
A_3 & \lesssim \sup_{N+1 \leq k <\infty } 2^{-\frac{k(1-\beta)}{r}} \bigg(\sum_{i=N+1}^k \int_{x_{i-1}}^{x_{i}} \bigg(\int_t^{x_{i}} {\mathcal W}^{\frac{\beta q}{r}} u \bigg)^{\frac{p}{p-q}}  d\big[V_p(a,t)^{\frac{pq}{p-q}}\big]   \bigg)^{\frac{p-q}{pq}}\\
& \quad + \sup_{N+2 \leq k <\infty} 2^{-\frac{k(1-\beta)}{r}}  \bigg(\sum_{j=N+1}^{k-1} \bigg(\int_{x_{j}}^{x_{j+1}} {\mathcal W}^{\frac{\beta q}{r}} u \bigg) \bigg(\int_{a}^{x_j}  d\big[V_p(a,t)^{\frac{pq}{p-q}}\big]   \bigg)^{\frac{p-q}{p}} \bigg)^{\frac{1}{q}} \notag\\
& \qquad +  \sup_{N+1 \leq k <\infty} 2^{-\frac{k(1-\beta)}{r}} \lim_{t \rightarrow a+} \bigg(\int_t^{x_{k}} {\mathcal W}^{\frac{\beta q}{r}} u \bigg)^{\frac{1}{q}}  V_p(a,t). \end{align*}
Further, \eqref{sup-sum} yields, 
\begin{align*}
A_3 & \lesssim 
\sup_{N+1 \leq k <\infty} 2^{-\frac{k(1-\beta)}{r}} \bigg( \int_{x_{k-1}}^{x_{k}} \bigg(\int_t^{x_{k}} {\mathcal W}^{\frac{\beta q}{r}} u \bigg)^{\frac{p}{p-q}}  d\big[V_p(a,t)^{\frac{pq}{p-q}}\big]   \bigg)^{\frac{p-q}{pq}} \notag\\
& \quad + \sup_{N+1 \leq k <\infty} 2^{-\frac{k(1-\beta)}{r}} \bigg(\int_{x_{k}}^{x_{k+1}} {\mathcal W}^{\frac{\beta q}{r}} u \bigg)^{\frac{1}{q}} \bigg( \int_{a}^{x_{k}}  d\big[V_p(a,t)^{\frac{pq}{p-q}}\big]\bigg)^{\frac{p-q}{pq}} \notag\\
& \qquad +  \sup_{N+1 \leq k <\infty } 2^{-\frac{k(1-\beta)}{r}} \lim_{t \rightarrow a+} \bigg(\int_t^{x_{k}} {\mathcal W}^{\frac{\beta q}{r}} u \bigg)^{\frac{1}{q}}  V_p(a,t)\\
& =: A_{3,1} + A_{3,2} + A_{3,3}.
\end{align*}
Integrating by parts again, we have that 
\begin{align*}
A_{3,1} & \lesssim \sup_{N+1 \leq k <\infty} 2^{-\frac{k(1-\beta)}{r}} \bigg(\int_{x_{k-1}}^{x_{k}} \bigg(\int_t^{x_{k}} {\mathcal W}^{\frac{\beta q}{r}} u \bigg)^{\frac{q}{p-q}} {\mathcal W}(t)^{\frac{\beta q}{r}} u(t)  V_p(a,t)^{\frac{pq}{p-q}} dt    \bigg)^{\frac{p-q}{pq}}. 
\end{align*}
Thus, \eqref{V-cut} gives,
\begin{align}\label{A_31<A3*+B1*}
A_{3,1} & \lesssim \sup_{N+1 \leq k <\infty} 2^{-\frac{k(1-\beta)}{r}} \bigg(\int_{x_{k-1}}^{x_{k}} \bigg(\int_t^{x_{k}} {\mathcal W}^{\frac{\beta q}{r}} u \bigg)^{\frac{q}{p-q}} {\mathcal W}(t)^{\frac{\beta q}{r}} u(t)  V_p(x_{k-1},t)^{\frac{pq}{p-q}} dt  \bigg)^{\frac{p-q}{pq}} \notag\\
& \quad + \sup_{N+2 \leq k <\infty} 2^{-\frac{k(1-\beta)}{r}} \bigg(\int_{x_{k-1}}^{x_{k}} {\mathcal W}^{\frac{\beta q}{r}} u \bigg)^{\frac{1}{q}}  V_p(a,x_{k-1}) \notag\\
& \lesssim \mathcal{A}_3 + \mathcal{B}_1. 
\end{align}
Additionally,
\begin{align}\label{A_32<A3*+B1*}
A_{3,2} \lesssim \sup_{N+1 \leq k <\infty} 2^{-\frac{k(1-\beta)}{r}} \bigg(\int_{x_{k}}^{x_{k+1}} {\mathcal W}^{\frac{\beta q}{r}} u \bigg)^{\frac{1}{q}} V_p(a,x_k)  \leq \mathcal{B}_1.
\end{align}

Lastly, we will find a suitable upper estimate for $A_{3,3}$. To this end, we will treat the cases $N= -\infty$ and $N > -\infty$, separately. Observe that, if $N=-\infty$, since $x_i \rightarrow a$ if $i\rightarrow -\infty$, we have for any $k$,
\begin{align}\label{lim<sup-1}
\lim_{t \rightarrow a+} \bigg(\int_t^{x_{k}} {\mathcal W}^{\frac{\beta q}{r}} u \bigg)^{\frac{1}{q}}  V_p(a,t) & = \lim_{i \rightarrow -\infty} \bigg(\int_{x_i}^{x_{k}} {\mathcal W}^{\frac{\beta q}{r}} u \bigg)^{\frac{1}{q}}  V_p(a,x_i) \notag \\
& \leq \sup_{i < k}  \bigg(\int_{x_i}^{x_{k}} {\mathcal W}^{\frac{\beta q}{r}} u \bigg)^{\frac{1}{q}} V_p(a,x_i).    
\end{align}
Then, then using \eqref{lim<sup-1} together with \eqref{3-sup-equiv}, we get
\begin{align*}
A_{3,3} & \lesssim  \sup_{k\in \mathbb{Z} } 2^{-\frac{k(1-\beta)}{r}} \sup_{i < k}  \bigg(\int_{x_i}^{x_{k}} {\mathcal W}^{\frac{\beta q}{r}} u \bigg)^{\frac{1}{q}}    V_p(a,x_i) \\
&\approx \sup_{k \in \mathbb{Z}} 2^{-\frac{k(1-\beta)}{r}} \bigg(\int_{x_{k-1}}^{x_{k}} {\mathcal W}^{\frac{\beta q}{r}} u \bigg)^{\frac{1}{q}} V_p(a,x_{k-1}) \lesssim \mathcal{B}_1. 
\end{align*}
On the other hand, if $N > -\infty$
\begin{align} \label{lim<sup-2}
 \lim_{t \rightarrow a+} \bigg(\int_t^{x_{k}} {\mathcal W}^{\frac{\beta q}{r}} u \bigg)^{\frac{1}{q}}  V_p(a,t) &\leq \esup_{t \in (a, x_{N+1})} \bigg(\int_t^{x_k} {\mathcal W}^{\frac{\beta q}{r}} u \bigg)^{\frac{1}{q}}  V_p(a,t) \notag\\
 &\approx \esup_{t \in (a, x_{N+1})} \bigg(\int_t^{x_{N+1}} {\mathcal W}^{\frac{\beta q}{r}} u \bigg)^{\frac{1}{q}}  V_p(a,t)\notag \\
 & \qquad +  \bigg(\int_{x_{N+1}}^{x_{k}} {\mathcal W}^{\frac{\beta q}{r}} u \bigg)^{\frac{1}{q}}  V_p(a,x_{N+1}).
\end{align}
Additionally, it is easy to see that
\begin{align}
\esup_{\tau \in (x, y)} \bigg(\int_{\tau}^y {\mathcal W}^{\frac{\beta q}{r}} u \bigg)^{\frac{1}{q}} V_p(x,\tau) \leq \bigg(\int_x^y \bigg(\int_t^y {\mathcal W}^{\frac{\beta q}{r}} u \bigg)^{\frac{q}{p-q}} {\mathcal W}(t)^{\frac{\beta q}{r}} u(t) V_p(x,t)^{\frac{pq}{p-q}} dt  \bigg)^{\frac{p-q}{pq}}. \label{sup-int-estimate}
\end{align}
First, using \eqref{lim<sup-2}, we get
\begin{align*}
A_{3,3} & \lesssim  \sup_{N+1 \leq k <\infty } 2^{-\frac{k(1-\beta)}{r}} \esup_{t \in (a, x_{N+1})} \bigg(\int_t^{x_{N+1}} {\mathcal W}^{\frac{\beta q}{r}} u \bigg)^{\frac{1}{q}}  V_p(a,t) \\
& \hspace{2cm} + \sup_{N+2 \leq k <\infty } 2^{-\frac{k(1-\beta)}{r}} \sup_{N+1 \leq i < k }\bigg(\int_{x_{i}}^{x_{k}} {\mathcal W}^{\frac{\beta q}{r}} u \bigg)^{\frac{1}{q}}  V_p(a,x_{i}).
\end{align*}
Then, using \eqref{sup-int-estimate} for the first term and applying \eqref{3-sup-equiv} for the second term, we get
\begin{align*}
A_{3,3} & \lesssim  \sup_{N+1 \leq k <\infty } 2^{-\frac{k(1-\beta)}{r}} \bigg(\int_{x_{k-1}}^{x_k} \bigg(\int_t^{x_k} {\mathcal W}^{\frac{\beta q}{r}} u \bigg)^{\frac{q}{p-q}} {\mathcal W}(t)^{\frac{\beta q}{r}} u(t) \,  V_p(a,t)^{\frac{pq}{p-q}}dt \bigg)^{\frac{p-q}{pq}} \\
& \hspace{2cm} + \sup_{N+2 \leq k <\infty } 2^{-\frac{k(1-\beta)}{r}} \bigg(\int_{x_{k-1}}^{x_k} {\mathcal W}^{\frac{\beta q}{r}} u \bigg)^{\frac{1}{q}} V_p(a, x_{k-1}).
\end{align*}
Finally, using \eqref{V-cut}, we arrive at
\begin{align*}
A_{3,3}  \lesssim  \mathcal{A}_3 + \sup_{N+2 \leq k <\infty } 2^{-\frac{k(1-\beta)}{r}} \bigg(\int_{x_{k-1}}^{x_k} {\mathcal W}^{\frac{\beta q}{r}} u \bigg)^{\frac{1}{q}} V_p(a, x_{k-1}) \lesssim \mathcal{A}_3 + \mathcal{B}_1.
\end{align*}
Consequently, we have for any $N \in \mathbb{Z}\cup \{-\infty\}$, $A_{3,3} \lesssim \mathcal{A}_3 + \mathcal{B}_1$. Combining the last estimate with \eqref{A_31<A3*+B1*} and \eqref{A_32<A3*+B1*}, we arrive at $A_3 \lesssim \mathcal{A}_3 + \mathcal{B}_1$, hence, $C \approx A_3 + \mathcal{B}_1$.

Let us now continue the proof by showing $A_3 + \mathcal{B}_1\approx C_1+ C_4$. 

To this end, taking $\alpha = \frac{1-\beta}{r}$ and 
\begin{equation*}
h(x) = \bigg(\int_a^{x} \bigg(\int_t^x {\mathcal W}^{\frac{\beta q}{r}} u \bigg)^{\frac{q}{p-q}} {\mathcal W}(t)^{\frac{\beta q}{r}} u(t) V_p(a,t)^{\frac{pq}{p-q}}dt \bigg)^{\frac{p-q}{pq}}, \quad x \in (a, b)
\end{equation*}
in \eqref{sup.equiv}, we have that $A_3 \approx C_4$. Moreover, we have already shown in \eqref{B1*<C1} and \eqref{C1<A1+B1*} that $\mathcal{B}_1 \lesssim C_1 \lesssim A_1 + \mathcal{B}_1$. Furthermore, \eqref{sup-int-estimate} yields that $A_1 \lesssim A_3$. Consequently, we have $A_3 + \mathcal{B}_1 \lesssim C_4 + C_1 \lesssim A_3 + \mathcal{B}_1$, which is the desired estimate. 

\rm(iv) Let $\max\{r,q \}< p$. Then, using [Theorem~\ref{T:disc.char.}, (iv)], we have that $C \approx \mathcal{B}_2 + \mathcal{A}_4$. First of all, we will show that $\mathcal{B}_2 + \mathcal{A}_4 \approx B_2 + A_4$, where  
\begin{equation*}
A_4 := \bigg( \sum_{k=N+1}^{\infty} 2^{-k\frac{p(1-\beta)}{p-r}} \bigg( \int_{a}^{x_{k}} \bigg(\int_{t}^{x_{k}} {\mathcal W}^{\frac{\beta q}{r}} u \bigg)^{\frac{q}{p-q}} {\mathcal W}(t)^{\frac{\beta q}{r}} u(t) V_p(a, t)^{\frac{pq}{p-q}} dt  \bigg)^{\frac{r(p-q)}{q(p-r)}}  \bigg)^{\frac{p-r}{pr}},
\end{equation*} 
and $B_2$ is defined in \eqref{B2}.

It is clear that $\mathcal{A}_4 \leq A_4$. We have already shown in \eqref{B2*<A2+B2} that $\mathcal{B}_2 \lesssim A_2 + B_2$.  Moreover, analogously as in the previous proof, using \eqref{sup-int-estimate}, one can easily see that $A_2 \lesssim A_4$. Thus, $\mathcal{B}_2 + \mathcal{A}_4 \lesssim B_2 + A_4$ follows.  

It remains to prove that $B_2 + A_4 \lesssim \mathcal{B}_2 + \mathcal{A}_4$. Assume that $\max\{\mathcal{A}_4, \mathcal{B}_2\}<\infty$. Then, using the same steps as in the previous case, we can see that \eqref{lim-0} holds; therefore, \eqref{IBP-estimate} is true in this case, as well. 

Applying \eqref{IBP-estimate} combined with \eqref{dec-sum-sum}, we obtain that
\begin{align*}
A_4 & \lesssim  \bigg( \sum_{k=N+1}^{\infty} 2^{-k\frac{p(1-\beta)}{p-r}} \bigg( \int_{x_{k-1}}^{x_{k}} \bigg(\int_{t}^{x_{k}} {\mathcal W}^{\frac{\beta q}{r}} u \bigg)^{\frac{p}{p-q}} d\bigg[ V_p(a, t)^{\frac{pq}{p-q}} \bigg]   \bigg)^{\frac{r(p-q)}{q(p-r)}}  \bigg)^{\frac{p-r}{pr}} \\
& \quad +  \bigg( \sum_{k=N+2}^{\infty} 2^{-k\frac{p(1-\beta)}{p-r}} \bigg(\int_{x_{k-1}}^{x_{k}} {\mathcal W}^{\frac{\beta q}{r}} u \bigg)^{\frac{pr}{q(p-r)}} \bigg( \int_{a}^{x_{k-1}} d\bigg[ V_p(a, t)^{\frac{pq}{p-q}} \bigg]  \bigg)^{\frac{r(p-q)}{q(p-r)}}  \bigg)^{\frac{p-r}{pr}}\\
& \qquad  + \bigg( \sum_{k=N+1}^{\infty} 2^{-k\frac{p(1-\beta)}{p-r}}  \bigg[\lim_{t \rightarrow a+} \bigg(\int_{t}^{x_{k}} {\mathcal W}^{\frac{\beta q}{r}} u \bigg)^{\frac{1}{q}} V_p(a, t)\bigg]^{\frac{pr}{p-r}} \bigg)^{\frac{p-r}{pr}} \\
& =: A_{4,1} + A_{4,2} + A_{4,3}
\end{align*}
holds.

As in the proof of the previous case, using integration by parts in combination with \eqref{V-cut}, we have that
\begin{align}
A_{4,1} &   \lesssim \bigg( \sum_{k=N+1}^{\infty} 2^{-k\frac{p(1-\beta)}{p-r}} \bigg( \int_{x_{k-1}}^{x_k} \bigg(\int_{t}^{x_k} {\mathcal W}^{\frac{\beta q}{r}} u \bigg)^{\frac{q}{p-q}} {\mathcal W}(t)^{\frac{\beta q}{r}} u(t) V_p(a,t)^{\frac{pq}{p-q}} dt  \bigg)^{\frac{r(p-q)}{q(p-r)}}  \bigg)^{\frac{p-r}{pr}} \notag\\
& \approx \mathcal{A}_4 + \bigg( \sum_{k=N+2}^{\infty} 2^{-k\frac{p(1-\beta)}{p-r}} \bigg( \int_{x_{k-1}}^{x_k}  {\mathcal W}^{\frac{\beta q}{r}} u  \bigg)^{\frac{pr}{q(p-r)}} V_p(a, x_{k-1})^{\frac{pr}{p-r}}   \bigg)^{\frac{p-r}{pr}} \notag \\
& \approx \mathcal{A}_4 + \bigg( \sum_{k=N+1}^{\infty} 2^{-k(1-\beta)} \bigg( \int_{x_{k}}^{x_{k+1}}  {\mathcal W}^{\frac{\beta q}{r}} u  \bigg)^{\frac{r}{q}}  2^{-k\frac{r(1-\beta)}{p-r}} \bigg( \int_{x_{k}}^{x_{k+1}}  {\mathcal W}^{\frac{\beta q}{r}} u  \bigg)^{\frac{r^2}{q(p-r)}} V_p(a, x_{k})^{\frac{pr}{p-r}}   \bigg)^{\frac{p-r}{pr}} \notag \\
& \lesssim \mathcal{A}_4 + \mathcal{B}_2.\label{A_41<A4+B2}
\end{align}
On the other hand, it is clear that
\begin{align}
A_{4,2} \lesssim \bigg( \sum_{k=N+2}^{\infty} 2^{-k\frac{p(1-\beta)}{p-r}} \bigg( \int_{x_{k-1}}^{x_k}  {\mathcal W}^{\frac{\beta q}{r}} u  \bigg)^{\frac{pr}{q(p-r)}} V_p(a, x_{k-1})^{\frac{pr}{p-r}}   \bigg)^{\frac{p-r}{pr}}
\lesssim \mathcal{B}_2. \label{A42<A4+B2}
\end{align}
Furthermore, if $N = -\infty$, using \eqref{lim<sup-1}, and then applying \eqref{3-sum-equiv}, we get 
\begin{align*}
A_{4,3} & \lesssim  \bigg( \sum_{k=-\infty}^{\infty} 2^{-k\frac{p(1-\beta)}{p-r}} \sup_{ i < k}  \bigg(\int_{x_{i}}^{x_k} {\mathcal W}^{\frac{\beta q}{r}} u \bigg)^{\frac{pr}{q(p-r)}}  V_p(a,x_i)^{\frac{pr}{p-r}} \bigg)^{\frac{p-r}{pr}}\\
& \approx 
\bigg( \sum_{k=-\infty}^{\infty} 2^{-k\frac{p(1-\beta)}{p-r}} \bigg(\int_{x_{k-1}}^{x_k} {\mathcal W}^{\frac{\beta q}{r}} u \bigg)^{\frac{pr}{q(p-r)}}  V_p(a,x_{k-1})^{\frac{pr}{p-r}} \bigg)^{\frac{p-r}{pr}}\lesssim \mathcal{B}_2.
\end{align*}
If $N >-\infty$, \eqref{lim<sup-2} together with \eqref{sup-int-estimate} yields, 
\begin{align*}
A_{4,3} & \lesssim \bigg( \sum_{k=N+1}^{\infty} 2^{-k\frac{p(1-\beta)}{p-r}} \bigg(\int_{x_{k-1}}^{x_k} \bigg(\int_t^{x_k} {\mathcal W}^{\frac{\beta q}{r}} u \bigg)^{\frac{q}{p-q}} {\mathcal W}(t)^{\frac{\beta q}{r}} u(t) \,  V_p(a,t)^{\frac{pq}{p-q}}dt \bigg)^{\frac{r(p-q)}{q(p-r)}} \bigg)^{\frac{p-r}{pr}} \\
& + \bigg( \sum_{k=N+2}^{\infty} 2^{-k\frac{p(1-\beta)}{p-r}}  \sup_{N+1 \leq i < k }\bigg(\int_{x_{i}}^{x_{k}} {\mathcal W}^{\frac{\beta q}{r}} u \bigg)^{\frac{pr}{q(p-r)}}  V_p(a,x_{i})^{\frac{pr}{p-r}} \bigg)^{\frac{p-r}{pr}}.
\end{align*}
Applying \eqref{V-cut} to the first term and \eqref{3-sum-equiv} to the second term, we have
\begin{align*}
A_{4,3} & \lesssim   \mathcal{A}_4 + \bigg( \sum_{k=N+2}^{\infty} 2^{-k\frac{p(1-\beta)}{p-r}} \bigg(\int_{x_{k-1}}^{x_k} {\mathcal W}^{\frac{\beta q}{r}} u \bigg)^{\frac{pr}{q(p-r)}} V_p(a,x_{k-1})^{\frac{pr}{p-r}} \bigg)^{\frac{p-r}{pr}}\lesssim \mathcal{A}_4 + \mathcal{B}_2.
\end{align*}
Thus, for any $N \in \mathbb{Z}\cup \{-\infty\}$, we arrive at $A_{4,3} \lesssim \mathcal{A}_4 + \mathcal{B}_2$. This fact, combined with \eqref{A_41<A4+B2} and \eqref{A42<A4+B2} yields $A_4 \lesssim \mathcal{A}_4 + \mathcal{B}_2$. Since $B_2 \leq \mathcal{B}_2$, we have $A_4 + B_2 \lesssim \mathcal{A}_4 + \mathcal{B}_2$ and consequently, $C\approx B_2 + A_4$. 

Now, we will show that $B_2 + A_4 \approx C_3 + C_5$.
Applying \eqref{int.equiv} with $\alpha = \frac{p(1-\beta)}{p-r}$ and 
\begin{equation*}
h(x) = \bigg( \int_{a}^{x} \bigg(\int_{t}^{x} {\mathcal W}^{\frac{\beta q}{r}} u \bigg)^{\frac{q}{p-q}} {\mathcal W}(t)^{\frac{\beta q}{r}} u(t) V_p(a,t)^{\frac{pq}{p-q}} dt \bigg)^{\frac{r(p-q)}{q(p-r)}}, \quad x\in (a, b),
\end{equation*}
it is clear that $A_4 \approx C_5$. We have also shown in \eqref{B2<C3} that $B_2 \lesssim C_3$. Hence, it remains to show that  $C_3 \lesssim A_4 + B_2$. To this end, we can use \eqref{sup-int-estimate}, and obtain $A_2 \lesssim A_4$. Moreover, we know from \eqref{C3<A2+B2} that $C_3 \lesssim A_2 + B_2$. Consequently, $C_3 \lesssim A_4+ B_2$ holds and the proof is complete. 
\qed

\

\textbf{Proof of Theorem~\ref{Cor}}
We will prove that inequality \eqref{monot.ineq-ab} holds for all $f \in \mathfrak{M}^{\uparrow}(a,b)$ if and only if inequality
\begin{equation}\label{special}
\bigg(\int_a^b \bigg(\int_a^x \bigg(\int_a^t h(\tau) d\tau \bigg)^\frac{1}{p}  u(t) dt \bigg)^{q} w(x) dx  \bigg)^{\frac{p}{q}} \leq C^p \int_a^b h(x) \bigg(\int_x^b v\bigg) dx
\end{equation}
holds for all $h \in \mathfrak{M}^+(a,b)$. 

Assume that \eqref{monot.ineq-ab} holds for all $f \in \mathfrak{M}^{\uparrow}(a,b)$. Substituting $f(x)= \big(\int_a^x h\big)^{\frac{1}{p}}$, $x\in (a,b)$ for $h \in \mathfrak{M}^+(a,b)$ in \eqref{monot.ineq-ab} and applying Fubini's theorem on the right-hand side, \eqref{special} follows. 

Conversely, assume that \eqref{special} holds for all $h \in \mathfrak{M}^+(a,b)$.  Since any $f \in \mathfrak{M}^{\uparrow}(a,b)$, even if $f(0)>0$, can be approximated pointwise from below by a function of the form $f(x)^p=\int_a^x h$, $x\in (a,b)$, then  the validity of \eqref{special}  yields \eqref{monot.ineq-ab}. Therefore, the result follows from Theorem~\ref{T:main}.

\qed

\textbf{Acknowledgments}

T.~\"{U}nver would like to express her gratitude to the Institute of Mathematics of the Czech Academy of Sciences for hosting her, providing her with a perfect working atmosphere, and supplying technical support.

We would like to thank the anonymous referees for their thorough and critical reading of the paper and their numerous useful comments and suggestions, which helped improve it substantially.

\section*{Declarations}

\begin{itemize}
\item Funding:

The research of  A.~Gogatishvili 
was partially supported by  the grant project 23-04720S of the Czech Science Foundation (GA\v{C}R),  The Institute of Mathematics, CAS is supported  by RVO:67985840, by  Shota Rustaveli National Science Foundation (SRNSF), grant no: FR21-12353, and by the
grant Ministry of Education and Science of the Republic of Kazakhstan (project no.
AP14869887).
The research of T.~\"{U}nver was supported by the grant of The Scientific and Technological Research Council of Turkey (TUBITAK), Grant No: 1059B192000075.

\item Conflict of interest:

The authors have no competing interests to declare relevant to this article's content.

\end{itemize}

\end{document}